\documentclass[11pt]{article}
\usepackage{amsfonts}
\usepackage{amsmath}
\usepackage{enumerate}
\usepackage[parfill]{parskip}
\usepackage{color}

\newtheorem{theorem}{Theorem}

\newtheorem{assumption}{Assumption}

\newtheorem{corollary}{Corollary}

\newtheorem{definition}{Definition}
\newtheorem{example}{Example}

\newtheorem{lemma}{Lemma}

\newtheorem{remark}[theorem]{Remark}

\newcommand{\mcf}{\mathcal{F}}
\newcommand{\mbf}{\mathbb{F}}
\newcommand{\mcg}{\mathcal{G}}
\newcommand{\mbg}{\mathbb{G}}
\newcommand{\mch}{\mathcal{H}}
\newcommand{\mbh}{\mathbb{H}}
\newcommand{\gs}{\sigma}
\newcommand{\var}{\text{Var}}
\newcommand{\gep}{\varepsilon}
\newcommand{\law}{\stackrel{\mathcal{L}}{=}}
\newcommand{\bee}{\begin{equation}}
\newcommand{\eee}{\end{equation}}
\newcommand{\ga}{\alpha}

\newenvironment{proof}[1][Proof]{\textbf{#1.} }{\ \rule{0.5em}{0.5em}}
\begin{document}

\title{On Progressive Filtration Expansions with a Process; Applications to Insider Trading}
\author{Younes Kchia\thanks{%
Goldman Sachs, Hong Kong; younes.kchia@polytechnique.org} \and %
Philip Protter\thanks{%
Statistics Department, Columbia University, New York, NY, 10027; pep2117@columbia.edu} 
\thanks{%
Supported in part by NSF grant DMS-1308483}} 
\date{\today}
\maketitle

\begin{abstract}
In this paper we study progressive filtration expansions with c\`adl\`ag processes. Using results from the theory of the weak convergence of $\sigma$-fields, we first establish a semimartingale convergence theorem. Then we apply it in a filtration expansion with a process setting and provide sufficient conditions for a semimartingale of the base filtration to remain a semimartingale in the expanded filtration. Applications to the expansion of a Brownian filtration are given. The paper concludes with applications to models of insider trading in financial mathematics.
\end{abstract}


\section{Introduction}

One of the key insights of K. It\^o when he developed the It\^o integral was to restrict the space of integrands to what we now call predictable processes.  This allowed the integral to have a type of bounded convergence theorem that N. Wiener was unable to obtain with unrestricted random integrands. The It\^o integral has since been extended to general semimartingales.  If one tries however to expand (i.e. to enlarge) the filtration, then one is playing with fire, and one may lose the key properties It\^o originally obtained with his restriction to predictable processes.  In the 1980's a theory of such filtration expansions  was nevertheless successfully developed for two types of expansion:  initial expansions and progressive expansions; see for instance \cite{Jacod:1987} and \cite{Jeulin/Yor:1985}, or the more recent partial exposition in~\cite[Chapter VI]{Protter:2005}. The initial expansion of a filtration $\mathbb F = (\mathcal F_t)_{t\geq 0}$ with a random variable $\tau$ is the filtration $\mathbb H$ obtained as the right-continuous modification of $(\mathcal F_t \vee \sigma (\tau))_{t\geq 0}$. The progressive expansion $\mathbb G$ is obtained as any right-continuous filtration containing $\mathbb F$ and making $\tau$ a stopping time. When referring to the progressive expansion with a random variable in this paper, we mean the smallest such filtration. One is usually interested in the cases where $\mathbb F$~semimartingales remain semimartingales in the expanded filtrations and in their decompositions when viewed as semimartingales in the expanded filtrations.  The theory of the expansion of filtrations has proved useful and of continuing interest in abstract probability theory (see for example the papers~\cite{AKH,Baudoin2,KN,ProtterFollmer,FWY,Jeanblanc/LeCam:2009,KLP,MansuyYor,NN,NY} and especially the recent thesis of A. Aksamit~\cite{AA}). 

The subject has regained interest recently, due to applications in Mathematical Finance. The work of A. Kyle~\cite{Kyle} in 1985 and of K. Back~\cite{Back1},\cite{Back2} in the early 1990's laid the foundation for a theory of the modeling of insider trading via a filtration of expansions approach. More recent work in the area includes the papers~\cite{AI,Baudoin1,Baudoin3,Baudoin4,BO,CV,CJN,FJS,GP,Imkeller,IPW,Wu,JZ}.

In this article we go beyond the simple cases of initial expansion and progressive expansion with a random variable.  Instead we consider the (more complicated) case of expansion of a filtration through dynamic enlargement, by adding a stochastic processes as it evolves simultaneously to the evolution of the original process.  In order to do this, we begin with simple cases where we add marked point processes, and then we use the theory of the convergence of $\sigma$ fields recently developed by Antonelli, Coquet, Kohatsu-Higa, Mackevicius, M\'emin, and Slominski (see~\cite{AKH},\cite{Coquet},\cite{Coquet0}) to obtain more sophisticated enlargement possibilities.  We combine the convergence results with an extension of an old result of Barlow and Protter~\cite{BarlowProtter}, finally obtaining the key results, which include the forms of the semimartingale decompositions in the enlarged filtrations.  We then apply these results to an example where we enlarge the filtration with another process which is evolving backwards in time.  To do this we need to use density estimates inspired by the work of Bally and Talay~\cite{BallyTalay}. Finally we conclude in a section where we develop some applications of our results to the financial theory of the modeling of insider trading. Here we build upon much preliminary work already done in the area.

The techniques developed in this paper require a long preliminary treatment of the convergence of $\sigma$ fields, and to a lesser extent the convergence of filtrations.  This delays the key theorems such that they occur rather late in the paper, so perhaps it is wise to indicate that the main results of interest (in the authors' opinion) are Theorems~\ref{T:exp} and~\ref{T:semimgX}, which show how one can expand filtrations with processes and have semimartingales remain semimartingales in the enlarged filtrations.  The authors also wish to mention here that the example provided in Theorem~\ref{T:diff} shows how the hypotheses  (perhaps a bit strange at first glance) of Theorem~\ref{T:semimgX}  can arise naturally in applications, and it shows the potential utility of the results of this paper.   That said, the preliminary results on the weak convergence of $\sigma$ fields has an interest in their own right.  

At the suggestion of a referee of a previous version of this paper, we have added a section (Section~\ref{section5}) where we apply our results to models of insider trading. In this section we also develop a concrete example of insider trading via high frequency trading. This builds on an already existing theory, developed by a variety of researchers, and relevant papers include~\cite{AFK,AI,Back1,Back2,Baudoin3,BO,CV,EJ,GP,Imkeller,IPW,KH,Kyle,Wu,JZ}, which is by no means an exhaustive list. With the plethora of recent scandals, such an addition seems timely.

\subsection{Previous Results}

For the initial expansion $\mathbb H$ of a filtration $\mathbb F = (\mathcal F_t)_{t\geq 0}$ with a random variable $\tau$, 
one well-known situation where $\mathbb F$ semimartingales remain semimartingales in the expanded filtration is when \textit{Jacod's criterion} is satisfied (see \cite{Jacod:1987} or alternatively~\cite[Theorem 10, p. 371]{Protter:2005}), and as far as one is concerned by the progressive expansion, filtration~$\mathbb G$, this always holds up to the random time $\tau$ as proved by Jeulin and Yor and holds on all $[0,\infty)$ for honest times (see \cite{Jeulin/Yor:1985}). In both \cite{KLP} and \cite{Jeanblanc/LeCam:2009}, this is proved to hold also for random times satisfying \textit{Jacod's criterion}. In \cite{KLP}, the authors link the two previous types of expansions and are able to provide similar results for more general types of expansion of filtrations. They extend for instance these results to the multiple time case, without any restrictions on the ordering of the individual times and more importantly to the filtration expanded by a counting process $N_t^n = \sum_{i=1}^n X_i 1_{\{\tau_i \leq t\}}$, i.e.~the smallest right-continuous filtration containing $\mathbb F$ and to which the process $N^n$ is adapted. 

For a given filtration $\mathbb F$ and a given c\`adl\`ag process $X$, the smallest right-continuous filtration containing $\mathbb F$ and to which $X$ is adapted will be called the progressive expansion of $\mathbb F$ with $X$. In this paper we pursue the analysis started in \cite{KLP} and investigate the stability of the semimartingale property of $\mathbb F$~semimartingales in progressive expansions of $\mathbb F$ with c\`adl\`ag processes $X$. We apply the results in \cite{KLP} together with results from the theory of weak convergence of $\sigma$-fields (see \cite{Coquet} and \cite{Coquet0}) to obtain a general criterion that guarantees this property, at least for $\mathbb F$~semimartingales satisfying suitable integrability assumptions. Hoover~\cite{Hoover}, following remarks by M. Barlow and S. Jacka, introduced the weak convergence of $\sigma$-fields and of filtrations in 1991.  The next big step was in 2000 with the seminal paper of Antonelli and Kohatsu-Higa\cite{AKH}.  This was quickly followed by the work of Coquet, M\'emin and Mackevicius \cite{Coquet0} and by Coquet, M\'emin and Slominsky \cite{Coquet}. We will recall fundamental results on the topic but we refer the interested reader to \cite{Coquet} and \cite{Coquet0} for details. In these papers, all filtrations are indexed by a compact time interval $[0,T]$. We work within the same framework and assume that a probability space $(\Omega, \mathcal H, P)$ and a positive integer $T$ are given. All filtrations considered in this paper are assumed to be completed by the $P$-null sets of $\mathcal H$. By the natural filtration of a process $X$, we mean the right-continuous filtration associated to the natural filtration of $X$. The concepts of weak convergence of $\sigma$-fields and of filtrations rely on the topology imposed on the space of c\`adl\`ag processes and we use the Skorohod $J_1$ topology as it is done in~\cite{Coquet}. 

\subsection{Outline}

An outline of this paper is the following. In section~\ref{S:1}, we recall basic facts on the weak convergence of $\sigma$-fields and establish fundamental lemmas for subsequent use. The last subsection provides a sufficient condition for the semimartingale property to hold for a given c\`adl\`ag adapted process based on the weak convergence of $\sigma$-fields. The sufficient condition we provide at this point is unlikely to hold in a filtration expansion context, however the proof of this result underlines what can go wrong under the more natural assumptions considered in the next section.

Section~\ref{section3} extends the main theorem in \cite{BarlowProtter} and proves a general result on the convergence of $\mathbb G^n$~special semimartingales to a $\mathbb G$~adapted process $X$, where $(\mathbb G^n)_{n\geq 1}$ and $\mathbb G$ are filtrations such that $\mathcal G^n_t$ converges weakly to $\mathcal G_t$ for each $t\geq 0$. The process $X$ is proved to be a $\mathbb G$~special semimartingale under sufficient conditions on the regularity of the local martingale and finite variation parts of the $\mathbb G^n$~semimartingales. This is then applied to the case where the filtrations $\mathbb G^n$ are obtained by progressively expanding a base filtration $\mathbb F$ with processes $N^n$ converging in probability to some process $N$. We provide sufficient conditions for an $\mathbb F$~semimartingale to remain a $\mathbb G$~semimartingale, where $\mathbb G$ is the progressive expansion of $\mathbb F$ with $N$. Section~\ref{3bis} contains a useful little theorem (Theorem~\ref{T:5}) concerning a dynamic process expansion obtained through the use of a sequence of sequences of honest times.

Section~\ref{section4} applies the results obtained in Section~\ref{section3} to the case where the base filtration $\mathbb F$ is progressively expanded by a c\`adl\`ag process whose increments satisfy a \textit{generalized Jacod's criterion} with respect to the filtration $\mathbb F$ along \textit{some} sequence of subdivisions whose mesh tends to zero. An application to the expansion of a Brownian filtration with a time reversed diffusion is given through a detailed study, and the canonical decomposition of the Brownian motion in the expanded filtration is provided. Finally, we apply our results to models of insider trading and provide several concrete examples in Section~\ref{section5}.

\subsection*{Acknowledgements:} The authors wish to thank Monique Jeanblanc for pointing out the seminal work of  Imkeller~\cite{Imkeller} and Zwierz~\cite{JZ}. They are also grateful to Jean Jacod, Umut \c{C}etin, and  Nizar Touzi for valuable help and advice. The first author wants to thank the hospitality of both Cornell University and Columbia University. The second author wishes to thank INRIA at Sophia-Antipolis and also the Courant Institute of NYU for their hospitality during a Sabbatical leave from Columbia University. 

\section{Weak convergence of $\sigma$-fields and filtrations} \label{S:1}

\subsection{Definitions and fundamental results}

Let $\mathbb{D}$ be the space of c\`adl\`ag\footnote{French acronym for right-continuous with left limits} functions from $[0,T]$ into $\mathbb{R}$. Let $\Lambda$ be the set of time changes from $[0,T]$ into $[0,T]$, i.e.~the set of all continuous strictly increasing functions $\lambda : [0,T]\rightarrow [0,T]$ such that $\lambda(0)=0$ and $\lambda(T)=T$. We define the Skorohod distance as follows
$$
d_S(x,y)=\inf_{\lambda \in \Lambda} \big\{ ||\lambda - Id||_{\infty} \vee ||x-y\circ\lambda||_{\infty}\big\}
$$
for each $x$ and $y$ in $\mathbb{D}$. Let $(X^n)_{n\geq 1}$ and $X$ be c\`adl\`ag processes (i.e.~whose paths are in~$\mathbb{D}$), indexed by $[0,T]$ and defined on $(\Omega, \mathcal H, P)$. We will write $X^n \stackrel{P}{\rightarrow} X$ when $(X^n)_{n\geq 1}$ converges in probability under the Skorohod $J_1$ topology to $X$ i.e.~when the sequence of random variables $(d_S(X^n,X))_{n\geq 1}$ converges in probability to zero. We can now introduce the concepts of weak convergence of $\sigma$-fields and of filtrations.

\begin{definition}\label{D:sigmafield}
A sequence of $\sigma$-fields $\mathcal{A}^n$ converges weakly to a $\sigma$-field $\mathcal{A}$ if and only if for all $B\in\mathcal{A}$, $E(1_B\mid\mathcal{A}^n)$ converges in probability to $1_B$. We write $\mathcal{A}^n\stackrel{w}{\rightarrow}\mathcal{A}$.
\end{definition}

\begin{definition}\label{D:filtration}
A sequence of right-continuous filtrations $\mathbb F^n$ converges weakly to a filtration $\mathbb F$ if and only if for all $B\in\mathcal F_T$, the sequence of c\`adl\`ag martingales $E(1_B\mid\mathcal F^n_{.})$ converges in probability under the Skorohod $J_1$ topology on $\mathbb{D}$ to the martingale $E(1_B\mid\mathcal F_{.})$. We write $\mathbb F^n\stackrel{w}{\rightarrow}\mathbb F$.
\end{definition}

The following lemmas provide characterizations of the weak convergence of $\sigma$-fields and filtrations. We refer to \cite{Coquet} for the proofs.

\begin{lemma}\label{L:sigma}
A sequence of $\sigma$-fields $\mathcal A^n$ converges weakly to a $\sigma$-field $\mathcal A$ if and only if $E(Z\mid\mathcal A^n)$ converges in probability to $Z$ for any integrable and $\mathcal A$~measurable random variable $Z$.
\end{lemma}

\begin{lemma}
A sequence of filtrations $\mathbb F^n$ converges weakly to a filtration $\mathbb F$ if and only if $E(Z\mid\mathcal F^n_{.})$ converges in probability under the Skorohod $J_1$ topology to $E(Z\mid\mathcal F_{.})$, for any integrable, $\mathcal F_T$~measurable random variable $Z$.
\end{lemma}

The weak convergence of the $\sigma$-fields $\mathcal F^n_t$ to $\mathcal F_t$ for all $t$ does not imply the weak convergence of the filtrations $\mathbb F^n$ to $\mathbb F$. The reverse implication does not hold neither.

Coquet, M\'emin and Slominsky provide a characterization of weak convergence of filtrations when the limiting filtration is the natural filtration of some c\`adl\`ag process $X$, see Lemma 3 in \cite{Coquet}. We provide a similar result for weak convergence of $\sigma$-fields when the limiting $\sigma$-field is generated by some c\`adl\`ag process $X$.

\begin{lemma}\label{L:carac}
Let $X$ be a c\`adl\`ag process. Define $\mathcal A=\sigma(X_t, 0\leq t\leq T)$ and let $(\mathcal A^n)_{n\geq 1}$ be a sequence of $\sigma$-fields. Then $\mathcal A^n\stackrel{w}{\rightarrow}\mathcal A$ if and only if
$$
E(f(X_{t_1}, \ldots, X_{t_k})\mid\mathcal A^n)\stackrel{P}{\rightarrow}f(X_{t_1}, \ldots, X_{t_k})
$$
for all $k\in \mathbb N$, $t_1, \ldots, t_k$ points of a dense subset $\mathcal D$ of $[0,T]$ containing $T$ and for any continuous and bounded function $f : \mathbb R^k\rightarrow \mathbb R$.
\end{lemma}

\begin{proof}
Necessity follows from  the definition of the weak convergence of $\sigma$-fields. Let us prove the sufficiency. Let $A\in\mathcal A$ and $\gep>0$. There exists $k\in\mathbb N$ and $t_1, \ldots, t_k$ in $\mathcal D$ such that 
$$
E(|f(X_{t_1}, \ldots, X_{t_k})-1_A|)<\varepsilon.
$$
Let $\eta>0$. We need to show that $P(|E(1_A\mid\mathcal{A}^n)-1_A|\geq \eta)$ converges to zero.
\begin{align*}
P&(|E(1_A\mid\mathcal A^n)-1_A|\geq \eta)\leq P(|E(1_A\mid\mathcal A^n)-E(f(X_{t_1}, \ldots, X_{t_k})\mid\mathcal A^n)|\geq\frac{\eta}{3})\\
&+P(|E(f(X_{t_1}, \ldots, X_{t_k})\mid\mathcal A^n)-f(X_{t_1}, \ldots, X_{t_k})|\geq\frac{\eta}{3})+P(|f(X_{t_1}, \ldots, X_{t_k})-1_A|\geq\frac{\eta}{3})\\
&\leq \frac{6}{\eta}E(|f(X_{t_1}, \ldots, X_{t_k})-1_A|)+P(|E(f(X_{t_1}, \ldots, X_{t_k})\mid\mathcal A^n)-f(X_{t_1}, \ldots, X_{t_k})|\geq\frac{\eta}{3})\\
&\leq \frac{6}{\eta}\varepsilon+P(|E(f(X_{t_1}, \ldots, X_{t_k})\mid\mathcal A^n)-f(X_{t_1}, \ldots, X_{t_k})|\geq\frac{\eta}{3})
\end{align*}
where the second inequality follows from the Markov inequality. By assumption, there exists $N$ such that for all $n\geq N$,
$$
P(|E(f(X_{t_1}, \ldots, X_{t_k})\mid\mathcal A^n)-f(X_{t_1}, \ldots, X_{t_k})|\geq\frac{\eta}{3})\leq \varepsilon
$$
hence $P(|E(1_A\mid\mathcal A^n)-1_A|\geq \eta)\leq (\frac{6}{\eta}+1)\varepsilon$.
\end{proof}

In \cite{Coquet}, the authors provide cases where the weak convergence of a sequence of natural filtrations of given c\`adl\`ag processes is guaranteed. We provide here a similar result for point wise weak convergence of the associated $\sigma$-fields.

\begin{lemma}\label{L:conv}
Let $(X^n)_{n\geq 1}$ be a sequence of c\`adl\`ag processes converging in probability to a c\`adl\`ag process $X$. Let $\mathbb F^n$ and $\mathbb F$ be the natural filtrations of $X^n$ and $X$ respectively. Then $\mathcal F^n_t\stackrel{w}{\rightarrow}\mathcal F_t$ for all $t$ such that $P(\Delta X_t \neq 0)=0$.
\end{lemma}

\begin{proof}
Let $t$ be such that $P(\Delta X_t \neq 0)=0$. Since $X$ is c\`adl\`ag, there exists $k\in\mathbb N$, and $t_1,\ldots,t_k \leq t$ such that $P(\Delta X_{t_i} \neq 0)=0$, for all $1\leq i\leq k$. Let $f : \mathbb R^k\rightarrow \mathbb R$ be a continuous and bounded function. By Lemma \ref{L:carac}, it suffices to show that
$$
E(f(X_{t_1}, \ldots, X_{t_k})\mid\mathcal F^n_t)\stackrel{P}{\rightarrow}f(X_{t_1}, \ldots, X_{t_k})
$$
An application of Markov's inequality leads to the following estimate
\begin{align*}
P(|E(f(X_{t_1}, \ldots,& X_{t_k})\mid\mathcal F^n_t)-f(X_{t_1}, \ldots, X_{t_k})|\geq \eta)\\
&\leq P(|E(f(X_{t_1}, \ldots, X_{t_k})-f(X^n_{t_1}, \ldots, X^n_{t_k})\mid\mathcal F^n_t)|\geq \frac{\eta}{2})\\
&+P(|E(f(X^n_{t_1}, \ldots, X^n_{t_k})\mid\mathcal F^n_t)-f(X_{t_1}, \ldots, X_{t_k})|\geq \frac{\eta}{2})\\
&\leq \frac{4}{\eta}E(|f(X^n_{t_1}, \ldots, X^n_{t_k})-f(X_{t_1}, \ldots, X_{t_k})|)
\end{align*}
Since $X^n\stackrel{P}{\rightarrow}X$ and $P(\Delta X_{t_i} \neq 0)=0$, for all $1\leq i\leq k$, it follows that
\begin{equation}\label{eq:discr}
(X^n_{t_1}, \ldots X^n_{t_k})\stackrel{P}{\rightarrow}(X_{t_1}, \ldots X_{t_k})
\end{equation}
and hence $f(X^n_{t_1}, \ldots X^n_{t_k})$ converges in $L^1$ to $f(X_{t_1}, \ldots X_{t_k})$. This ends the proof of the lemma.
\end{proof}

For a given c\`adl\`ag process $X$, a time $t$ such that $P(\Delta X_t \neq 0)>0$ will be called a fixed time of discontinuity of $X$, and we will say that $X$ has no fixed times of discontinuity if $P(\Delta X_t\neq 0)=0$ for all $0\leq t\leq T$. Lemma \ref{L:conv} can be improved when the sequence $X^n$ is the discretization of the c\`adl\`ag process $X$ along some refining sequence of subdivisions $(\pi_n)_{n\geq 1}$ such that each fixed time of discontinuity of $X$ belongs to $\cup_n \pi_n$.

\begin{lemma}\label{L:discr}
Let $X$ be a c\`adl\`ag process. Consider a sequence of subdivisions $(\pi_n=\{t^n_k\}, n\geq~1)$ whose mesh tends to zero and let $X_n$ be the discretized process defined by $X_t^n = X_{t^n_k}$, for all $t_n^k \leq t < t^n_{k+1}$. Let $\mathbb F$ and $\mathbb F^n$ be the natural filtrations of $X$ and $X^n$. If each fixed time of discontinuity of $X$ belongs to $\cup_n \pi_n$, then $\mathcal F^n_t\stackrel{w}{\rightarrow}\mathcal F_t$, for all $t$.
\end{lemma}

\begin{proof}
The proof is essentially the same as that of Lemma \ref{L:conv}. Now, equation (\ref{eq:discr}) holds because the subdivision contains the discontinuity points of $X$. 
\end{proof}

We will also need the two following lemmas from the theory of weak convergence of $\sigma$-fields. The first result is proved in \cite{Coquet0} and the second one in \cite{Coquet}.

\begin{lemma}\label{L:vee}
Let $(\mathcal A^n)_{n\geq 1}$ and $(\mathcal B^n)_{n\geq 1}$ be two sequences of $\sigma$-fields that weakly converge to $\mathcal A$ and $\mathcal B$, respectively. Then
$$
\mathcal A^n\vee\mathcal B^n\stackrel{w}{\rightarrow}\mathcal A\vee\mathcal B
$$
\end{lemma}

\begin{lemma}\label{L:subset}
Let $(\mathcal A^n)_{n\geq 1}$ and $(\mathcal B^n)_{n\geq 1}$ be two sequences of $\sigma$-fields such that $\mathcal A^n \subset \mathcal B^n$ for all $n$. Let $\mathcal A$ be a $\sigma$-field. If $\mathcal A^n\stackrel{w}{\rightarrow}\mathcal A$ then $\mathcal B^n\stackrel{w}{\rightarrow}\mathcal A$.
\end{lemma}

As pointed out in \cite{Coquet}, the results in Lemmas \ref{L:vee} and \ref{L:subset} are not true as far as one is interested in weak convergence of filtrations.

\subsection{Approximation of a given stopping time}

Let $(\mathbb{G}^n)_{n\geq 1}$ be a sequence of right-continuous filtrations and let $\mathbb{G}$ be a right-continuous filtration such that $\mathcal{G}^n_t\stackrel{w}{\rightarrow}\mathcal{G}_t$ for all $t$. In order to obtain our filtration expansion results, we need a key theorem that guarantees the $\mathbb G$~semimartingale property of a limit of $\mathbb G^n$~semimartingales as in Theorem \ref{T:semimg2}. The following lemma, which permits to approximate any $\mathbb G$~bounded stopping time $\tau$ by a sequence of $\mathbb G^n$~stopping times, will be of crucial importance in the proof of Theorem \ref{T:semimg2}, Part (ii). We prove this result using successive approximations in the case where $\tau$ takes a finite number of values and show how this property is inherited by bounded stopping times. We do not study the general case (unbounded stopping times) since we are working on the finite time interval $[0,T]$.

\begin{lemma}\label{L:stop}
Let $(\mathbb{G}^n)_{n\geq 1}$ be a sequence of right-continuous filtrations and let $\mathbb{G}$ be a right-continuous filtration such that $\mathcal{G}^n_t\stackrel{w}{\rightarrow}\mathcal{G}_t$ for all $t$. Let $\tau$ be a bounded $\mathbb G$ stopping time. Then there exists $\phi : \mathbb{N}\rightarrow \mathbb{N}$ strictly increasing and a bounded sequence $(\tau_n)_{n\geq 1}$ such that the subsequence $(\tau_{\phi(n)})_{n\geq 1}$ converges in probability to $\tau$ and each $\tau_{\phi(n)}$ is a $\mathbb G^{\phi(n)}$~stopping time.
\end{lemma}

\begin{proof}
Let $\tau$ be a $\mathbb G$~stopping time bounded by $T$. Then there exists a sequence $\tau_n$ of $\mathbb G$ stopping times decreasing a.s. to $\tau$ and taking values in $\big\{\frac{k}{2^n}, k\in\{0, 1, \cdots, [2^nT]+1\}\big\}$. This is true since the sequence $\tau_n=\frac{[2^n\tau]+1}{2^n}$ obviously works. Hence $\tau_n$ takes a finite number of values. We claim that 

\textbf{Claim.} \textit{for each $n$, we can construct a sequence $(\tau_{n,m})_{m\geq 1}$ converging in probability to~$\tau_n$, and such that $\tau_{n,m}$ is a $\mathbb G^m$~stopping time, for each $m$.}

Assume we can do so and let $\eta>0$ and $\varepsilon>0$. Then for each $n$, $\lim_{m\rightarrow\infty}P(|\tau_{n,m}-\tau_n|>\frac{\eta}{2})=0$, i.e.~for each $n$ there exists $M_n$ such that for all $m\geq M_n$, $P(|\tau_{n,m}-\tau_n|>\frac{\eta}{2})\leq\frac{\varepsilon}{2}$. Define $\phi(1)=M_1$ and $\phi(n)=\max(M_n, \phi(n-1)+1)$ by induction. The application $\phi:\mathbb N\rightarrow\mathbb N$ is strictly increasing, and for each $n$, $\tau_{n,\phi(n)}$ is a $\mathbb G^{\phi(n)}$~stopping time and $P(|\tau_{n,\phi(n)}-\tau_n|>\frac{\eta}{2})\leq\frac{\varepsilon}{2}$. It follows that
$$
P(|\tau_{n,\phi(n)}-\tau|>\eta)\leq P(|\tau_{n,\phi(n)}-\tau_n|>\frac{\eta}{2})+P(|\tau_n-\tau|>\frac{\eta}{2})\leq \frac{\varepsilon}{2}+P(|\tau_n-\tau|>\frac{\eta}{2})
$$
Since $\tau_n$ converges to $\tau$, there exists some $n_0$, such that for all $n\geq n_0$, $P(|\tau_n-\tau|>\frac{\eta}{2})\leq \frac{\gep}{2}$. Hence $\tau_{n,\phi(n)}\stackrel{P}{\rightarrow}\tau$. So in order to prove the lemma, it only remains to prove the claim above. 

\textbf{Proof of the claim.} We drop the index $n$ and assume that $\tau$ is a $\mathbb G$~stopping time that takes a finite number of values $t_1,\cdots, t_M$.  Since $\mathbb G$ is right-continuous, $1_{\{\tau=t_i\}}$ is $\mathcal{G}_{t_i}$~measurable, and since by assumption, for all $i$, $\mathcal{G}^m_{t_i}\stackrel{w}{\rightarrow}\mathcal{G}_{t_i}$, it follows that for all $i$
$$
E(1_{\{\tau=t_i\}}\mid\mathcal{G}^m_{t_i})\stackrel{P}{\rightarrow}1_{\{\tau=t_i\}}
$$
Now for $i=1$, we can extract a subsequence $E(1_{\{\tau=t_1\}}\mid\mathcal{G}^{\phi_1(m)}_{t_1})$ converging to $1_{\{\tau=t_1\}}$~a.s. and any sub-subsequence will also converge to $1_{\{\tau=t_1\}}$ a.s. Also, $\mathcal{G}^{\phi_1(m)}_{t_2}\stackrel{w}{\rightarrow}\mathcal{G}_{t_2}$, hence $E(1_{\{\tau=t_2\}}\mid\mathcal{G}^{\phi_1(m)}_{t_2})\stackrel{P}{\rightarrow}1_{\{\tau=t_2\}}$, and we can extract a further subsequence $E(1_{\{\tau=t_2\}}\mid\mathcal{G}^{\phi_1(\phi_2(m))}_{t_2})$ that converges a.s. to $1_{\{\tau=t_2\}}$. Since we have a finite number of possible values, we can repeat this reasoning up to time $t_M$. Define then $\phi=\phi_1 \circ \phi_2 \circ \cdots \circ \phi_n$, we get for all $i\in\{1,\cdots, M\}$, 
$$
E(1_{\{\tau=t_i\}}\mid\mathcal{G}^{\phi(m)}_{t_i})\stackrel{a.s}{\rightarrow}1_{\{\tau=t_i\}}.
$$
Define $\tau_m=\min_{\{i\mid E(1_{\{\tau=t_i\}}\mid\mathcal{G}^{m}_{t_i})>\frac{1}{2}\}}t_i$. Then
$$
\{\tau_m=t_i\}=\{E(1_{\{\tau=t_i\}}\mid\mathcal{G}^{m}_{t_i})>\frac{1}{2}\}\cap\{\forall t_j<t_i, E(1_{\{\tau=t_j\}}\mid\mathcal{G}^{m}_{t_j})\leq\frac{1}{2}\}
$$
and hence $\tau_m$ is a $\mathbb G^m$~stopping time. Also, obviously, $\tau_{\phi(m)}\stackrel{a.s}{\rightarrow}\tau$, hence $\tau_m\stackrel{P}{\rightarrow}\tau$.
\end{proof}

\subsection{Weak convergence of $\sigma$-fields and the semimartingale property}

Assume we are given a sequence of filtrations $(\mathbb F^m)_{m\geq 1}$ and define the filtration $\tilde{\mathbb F}=(\tilde{\mathcal F}_t)_{0\leq t\leq T}$, where $\tilde{\mathcal F}_t=\bigvee_{m}\mathcal F^m_t$. We prove in this section a stability result for $\tilde{\mathbb F}$~semimartingales. More precisely, we prove that if $X$ is an $\tilde{\mathbb F}$~semimartingale, then it remains an $\mathbb F$~semimartingale for any limiting (in the sense $\mathcal F^m_t\stackrel{w}{\rightarrow}\mathcal F_t$, for all $t \in [0,T]$) filtration $\mathbb F$ to which it is adapted.

The crucial tool for proving our first theorem is the Bichteler-Dellacherie characterization of semimartingales (see for example~\cite{Protter:2005}). Recall that if $\mathbb H$ is a filtration, an $\mathbb H$~predictable elementary process $H$ is a process of the form
$$
H_t(\omega) = \sum_{i=1}^{k}h_i(\omega)1_{]t_i,t_{i+1}]}(t);
$$
where $0 \leq t_1 \leq \ldots \leq t_{k+1} < \infty$, and each $h_i$ is $\mathcal H_{t_i}$~measurable. Moreover, for any $\mathbb H$~adapted c\`adl\`ag process $X$ and predictable elementary process $H$ of the above form, we write
$$
J_X(H) = \sum_{i=1}^{k} h_i(X_{t_{i+1}} - X_{t_{i}})
$$
\begin{theorem}[Bichteler-Dellacherie]\label{T:BichDell}
Let $X$ be an $\mathbb H$~adapted c\`adl\`ag process. Suppose that for every sequence $(H_n)_{n\geq 1}$ of bounded, $\mathbb H$ predictable elementary processes that are null outside a fixed interval $[0,N]$ and convergent to zero uniformly in $(\omega; t)$, we have that $\lim_{n\rightarrow\infty}J_X(H_n) = 0$ in probability. Then $X$ is an $\mathbb H$~semimartingale.
\end{theorem}

The converse is true by the Dominated Convergence Theorem for stochastic integrals. We can now state and prove the main theorem of this subsection.

\begin{theorem}\label{T:semimg}
Let $(\mathbb F^m)_{m\geq 1}$ be a sequence of filtrations. Let $\mathbb F$ be a filtration such that for all $t \in [0,T]$, $\mathcal F^m_t\stackrel{w}{\rightarrow}\mathcal F_t$. Define the filtration $\tilde{\mathbb F}=(\tilde{\mathcal F}_t)_{0\leq t\leq T}$, where $\tilde{\mathcal F}_t=\bigvee_{m}\mathcal F^m_t$. Let $X$ be an $\mathbb F$~adapted c\`adl\`ag process such that $X$ is an $\tilde{\mathbb F}$~semimartingale. Then $X$ is an $\mathbb F$~semimartingale.
\end{theorem}

\begin{proof}
For a fixed $N>0$, consider a sequence of bounded, $\mathbb F$ predictable elementary processes of the form
$$
H^n_t = \sum_{i=1}^{k_n}h_i^n1_{]t^n_i,t^n_{i+1}]}(t);
$$
null outside the fixed time interval $[0,N]$ and with $h^n_i$ being $\mathcal F_{t^n_i}$~measurable. Suppose that $H^n$ converges to zero uniformly in $(\omega, t)$. We prove that $ J_X(H^n) \stackrel{P}{\rightarrow} 0$.

For each $m$, define the sequence of bounded $\mathbb F^m$~predictable elementary processes
$$
H^{n,m}_t = \sum_{i=1}^{k_n}E(h_i^n\mid\mathcal{F}^m_{t^n_i})1_{]t^n_i,t^n_{i+1}]}(t);
$$
By assumption, $\mathcal F^m_t\stackrel{w}{\rightarrow}\mathcal F_t$ for all $0\leq t\leq T$. Hence for all $n$ and $1\leq i\leq k_n$, $\mathcal F^m_{t_i^n}\stackrel{w}{\rightarrow}\mathcal F_{t_i^n}$. Since $h^n_i$ is bounded (hence integrable) and $\mathcal F_{t_i^n}$~measurable, it follows from Lemma \ref{L:sigma} that $E(h^n_i\mid\mathcal F^m_{t^n_i})\stackrel{P}{\rightarrow}h^n_i$ and hence $E(h^n_i\mid\mathcal F^m_{t^n_i})(X_{t_{i+1}^n}-X_{t_i^n})\stackrel{P}{\rightarrow}h^n_i(X_{t^n_{i+1}}-X_{t^n_i})$ for each $n$ and $1\leq i\leq k_n$ since $(X_{t^n_{i+1}}-X_{t^n_i})$ is finite~a.s. 
Let $\eta>0$.
$$
P\Big(\big|J_X(H^{n,m})-J_X(H^n)\big|>\eta\Big) \leq \sum_{i=1}^{k_n}P\Big(\Big|\big(E(h^n_i\mid\mathcal F^m_{t^n_i})-h^n_i\big)\big(X_{t^n_{i+1}}-X_{t^n_i}\big)\Big|>\frac{\eta}{k_n}\Big)
$$
For each fixed $n$, the right side quantity converges to $0$ as $m$ tends to $\infty$. This proves that for each $n$,
$$
J_X(H^{n,m})\stackrel{P}{\rightarrow}J_X(H^n).
$$
Let $\delta>0$ and $\gep>0$. For each $n$ and $m$,
\begin{equation}\label{eq:exchangelim}
P(|J_X(H^n)|>\delta)\leq P\Big(\big|J_X(H^{n,m})-J_X(H^n)\big|>\frac{\delta}{2}\Big)+P(|J_X(H^{n,m})|>\frac{\delta}{2})
\end{equation}
From $J_X(H^{n,m})\stackrel{P}{\rightarrow}J_X(H^n)$, it follows that for each $n$, there exists $M^n_0$ such that for all $m\geq M^n_0$,
$$
P\Big(\big|J_X(H^{n,m})-J_X(H^n)\big|>\frac{\delta}{2}\Big)\leq \frac{\gep}{2}
$$
Hence $P(|J_X(H^n)|>\delta)\leq \frac{\gep}{2}+P(|J_X(H^{n,M^n_0})|>\frac{\delta}{2})$. First $E(h_i^n\mid\mathcal F^{M^n_0}_{t^n_i})$ is bounded, $\tilde{\mathcal F}_{t^n_i}$~measurable so that $H_t^{n,M^n_0}=\sum_{i=1}^{k_n}E(h_i^n\mid\mathcal F^{M^n_0}_{t^n_i})1_{]t^n_i,t^n_{i+1}]}(t)$ is a bounded $\tilde{\mathbb F}$~predictable process. Since $H^n$ converges to zero uniformly in $(\omega, t)$, it follows that $h^n_i$ converges to zero uniformly in $(\omega, i)$ so that there exists $n_0$ such that for each $n\geq n_0$, for all $(\omega, i)$, $|h^n_i(\omega)|\leq \gep$. Hence, for all $(\omega, t)$ and $n\geq n_0$
$$
|H^{n,M^n_0}_t(\omega)| \leq \sum_{i=1}^{k_n}E(|h_i^n|\mid\mathcal{F}^{M^n_0}_{t^n_i})(\omega)1_{]t^n_i,t^n_{i+1}]}(t)\leq \gep \sum_{i=1}^{k_n}1_{]t^n_i,t^n_{i+1}]}(t) \leq \gep
$$
Therefore $H^{n,M^n_0}$ is a sequence of bounded $\tilde{\mathbb F}$~predictable processes null outside the fixed interval $[0,N]$ that converges uniformly to zero in $(\omega, t)$. Since by assumption $X$ is a $\tilde{\mathbb F}$~semimartingale, it follows from the converse of Bichteler-Dellacherie's theorem that $J_X(H^{n,M^n_0})$ converges to zero in probability, hence, for $n$ large enough, $P(|J_X(H^{n,M^n_0})|>\frac{\delta}{2})\leq \frac{\gep}{2}$ and
$$
P(|J_X(H^n)|>\delta)\leq\gep
$$
Applying now Theorem \ref{T:BichDell} proves that $X$ is an $\mathbb F$~semimartingale.
\end{proof}

Let $X$ be an $\tilde{\mathbb F}$~semimartingale. Theorem \ref{T:semimg} proves that $X$ remains an $\mathbb F$~semimartingale for any limiting filtration $\mathbb F$ (in the sense $\mathcal F^m_t \stackrel{w}{\rightarrow} \mathcal F_t$ for all $0\leq t\leq T$) to which $X$ is adapted. Of course, if $\mathbb F\subset\tilde{\mathbb F}$, Stricker's theorem already implies that $X$ is an $\mathbb F$~semimartingale. But there is no general link between the filtration $\tilde{\mathbb F}=\bigvee_m\mathbb F^m$ and the limiting filtration~$\mathbb F$. A trivial example is given by taking $\mathbb F$ to be the trivial filtration (it can be seen from Definition \ref{D:sigmafield} that the trivial filtration satisfies $\mathcal F^m_t\stackrel{w}{\rightarrow}\mathcal F_t$, for all $t$, for any given sequence of filtrations $\mathbb F^m$). One can also have $\bigvee_m\mathbb F^m\subset\mathbb F$, as it is the case in the following important example.

\begin{example}\label{ex:dis}
Let $X$ be a c\`adl\`ag process. Consider a sequence of subdivisions $\{t_k^n\}$ whose mesh tends to zero and let $X^n$ be the discretized process defined by $X^n_t=X_{t^n_k}$, for all $t^n_k\leq t <t^n_{k+1}$. Let $\mathbb F$ and $\mathbb F^n$ be the natural filtrations of $X$ and $X^n$. It is well known that for all $t$, $\mathcal F_{t^{-}} \subset \bigvee_{n}\mathcal F^n_t\subset\mathcal F_t$. Also, $X^n$ converges a.s.~to the process $X$, hence $X^n\stackrel{P}{\rightarrow}X$. Assume now that $X$ has no fixed times of discontinuity. Then Lemma \ref{L:conv} guarantees that $\mathcal F^n_t\stackrel{w}{\rightarrow}\mathcal F_t$, for all $t$. Moreover, if $\mathbb F$ is left-continuous (which is usually the case, and holds for example  when $X$ is a c\`adl\`ag Hunt Markov process) then $\bigvee_{n}\mathcal F^n_t=\mathcal F_t$ for all $t$.
\end{example}

We provide now another example where $\bigvee_n\mathcal F^n_t$ is itself a limiting $\sigma$-field for $(\mathcal F^n_t)_{n\geq 1}$, for each~$t$.

\begin{example}
Assume that $\mathbb F^n$ is a sequence of filtrations such that for all $t$, the sequence of $\sigma$-fields $(\mathcal F^n_t)_{n\geq 1}$ is increasing for the inclusion. Define $\tilde{\mathcal F}_t=\bigvee_n\mathcal F^n_t$. Then for each $t$, $\mathcal F^n_t\stackrel{w}{\rightarrow}\tilde{\mathcal F}_t$. To see this, fix $t$ and let $X$ be an integrable $\tilde{\mathcal F}_t$~measurable random variable. Then $M_n=E(X\mid\mathcal F^n_t)$ is a closed martingale and the convergence theorem for closed martingales ensures that $M_n$ converges to $X$ in $L^1$, which implies that $E(X\mid\mathcal F^n_t)\stackrel{P}{\rightarrow}X$. Lemma \ref{L:sigma} allows us to conclude.
\end{example}

Checking in practice that X is an $\tilde{\mathbb F}$~semimartingale can be a hard task. In subsequent sections, we replace the strong assumption \textit{$X$ is an $\tilde{\mathbb F}$~semimartingale} by the more natural assumption \textit{$X$ is an $\mathbb F^n$~semimartingale, for each $n$}. Theorem \ref{T:semimg} is very instructive since we see from the proof what goes wrong under this new assumption : the change in the order of limits in (\ref{eq:exchangelim}) cannot be justified anymore and extra integrability conditions will be needed. They are introduced in the next section. 

This assumption arises naturally in filtration expansion theory in the following way. Assume we are given a base filtration $\mathbb F$ and a sequence of processes $N^n$ which converges (in probability for the Skorohod $J_1$ topology) to some process $N$. Let $\mathbb N^n$ and $\mathbb N$ be their natural filtrations and $\mathbb G^n$ (resp. $\mathbb G$) the smallest right-continuous filtration containing $\mathbb{F}$ and to which $N^n$ (resp. $N$) is adapted. Assume that for each $n$, every $\mathbb F$~semimartingale remains a $\mathbb G^n$~semimartingale. Does this property also hold between $\mathbb F$ and $\mathbb G$? In the next section we answer this question under the assumption of weak convergence of the $\sigma$-fields $\mathcal G^n_t$ to $\mathcal G_t$ for each $t$, for a class of $\mathbb F$~semimartingales $X$ satisfying some integrability conditions. If moreover $\mathbb G^n\stackrel{w}{\rightarrow}\mathbb G$, we are able to provide the $\mathbb G$~decomposition of such~$X$. 

\section{Filtration expansion with processes}\label{section3}

In preparation for treating the expansion of filtrations via processes, we need to establish a general result on the convergence of semimartingales, which is perhaps of interest in its own right.  

\subsection{Convergence of semimartingales}

The following theorem is a generalization of the main result in \cite{BarlowProtter}.

\begin{theorem}\label{T:semimg2}
Let $(\mathbb{G}^n)_{n\geq 1}$ be a sequence of right-continuous filtrations and let $\mathbb{G}$ be a filtration such that $\mathcal{G}^n_t\stackrel{w}{\rightarrow}\mathcal{G}_t$ for all $t$. Let $(X^n)_{n\geq 1}$ be a sequence of $\mathbb{G}^n$~semimartingales with canonical decomposition $X^n = X^n_0 + M^n + A^n$. Assume there exists $K>0$ such that for all $n$, 
$$
E(\int_0^T|dA^n_s|) \leq K \qquad \text{and} \qquad E(\sup_{0\leq s\leq T}|M^n_s|)\leq K
$$
Then the following holds.
\begin{itemize}
	\item[(i)] Assume there exists a $\mathbb{G}$~adapted process $X$ such that $E(\sup_{0\leq s\leq T}|X^n_s-X_s|)\rightarrow 0$. Then $X$ is a $\mathbb{G}$~special semimartingale.
	
	\item[(ii)] Moreover, assume $\mathbb G$ is right-continuous and let $X=M+A$ be the canonical decomposition of $X$. Then $M$ is a $\mathbb{G}$ martingale and $\int_0^T|dA_s|$ and $\sup_{0\leq s\leq T} |M_s|$ are integrable.
\end{itemize}
\end{theorem}

\begin{proof}\textbf{Part (i).} The idea of the proof of Part (i) is similar to the one in \cite{BarlowProtter}. First, $X$ is c\`adl\`ag since it is the a.s.~uniform limit of a subsequence of the c\`adl\`ag processes $(X^n)_{n\geq 1}$. Also, since $||X^n_0-X_0||_1\rightarrow 0$, we can take w.l.o.g $X_0^n=X_0=0$, and we do~so. The integrability assumptions guarantee that $E(\sup_s|X^n_s|)\leq 2K$ and up to replacing $K$ by $2K$, we assume that $E(\sup_s|M^n_s|)\leq K$, $E(\int_0^T|dA^n_s|)\leq K$ and $E(\sup_s|X^n_s|)\leq K$. Then $E(\sup_s|X_s|)\leq E(\sup_s|X_s-X^n_s|)+K$ and by taking limits $E(\sup_s|X_s|)\leq K$.

Let $H$ be a $\mathbb{G}$~predictable elementary process of the form $H_t=\sum_{i=1}^kh_i1_{]t_i,t_{i+1}]}(t)$, where $h_i$ is a $\mathcal{G}_{t_i}$~measurable random variable such that $|h_i|\leq 1$ and $t_1<\ldots <t_k<t_{k+1}=T$. Define now $H^n_t=\sum_{i=1}^kh^n_i1_{]t_i,t_{i+1}]}(t)$, where $h^n_i=E(h_i\mid\mathcal{G}^n_{t_i})$. Then $h^n_i$ is a $\mathcal{G}^n_{t_i}$~measurable random variable satisfying $|h^n_i|\leq 1$, hence $H^n$ is a bounded $\mathbb{G}^n$~predictable elementary process. It follows that $H^n\cdot M^n$ is a $\mathbb{G}^n$~martingale and for each $n$,
$$
|E((H^n\cdot X^n)_T)| \leq |E(\int_0^TH^n_sdA^n_s)|\leq E(\int_0^{T}|dA^n_s|)\leq K
$$
Therefore, for each $n$,
\begin{equation}\label{eq:lim}
|E((H\cdot X)_T)|\leq |E((H \cdot X)_T-(H^n\cdot X^n)_T)|+K
\end{equation}
Since $h_i$ is $\mathcal{G}_{t_i}$~measurable and $\mathcal{G}_t^n\stackrel{w}{\rightarrow}\mathcal{G}_t$ for all $t$, $h^n_i\stackrel{P}{\rightarrow}h_i$ for all $1\leq i\leq k$. Since the set $\{1,\ldots,k\}$ is finite, successive extractions allow us to find a subsequence $\psi(n)$ (independent from $i$) such that for each $1\leq i\leq k$, $h^{\psi(n)}_i$ converges a.s.~to $h_i$. So up to working with the $\mathbb{G}^{\psi(n)}$ predictable elementary processes $H^{\psi(n)}$ and the stochastic integrals $H^{\psi(n)}\cdot X^{\psi(n)}$ in (\ref{eq:lim}), we can assume that $h^n_i$ converges a.s.~to $h_i$, for each $1\leq i\leq k$.

Now, $|E((H\cdot X)_T-(H^n\cdot X^n)_T)|\leq \sum_{i=1}^kE(|h_iY_i-h^n_iY^n_i|)$ where $Y_i=X_{t_{i+1}}-X_{t_i}$ and $Y^n_i=X^n_{t_{i+1}}-X^n_{t_i}$. Each term in the sum can be bounded as follows.
\begin{align*}
E(|h_iY_i-&h^n_iY^n_i|)\leq E(|Y^n_i(h^n_i-h_i)|)+E(|h_i(Y^n_i-Y_i)|)\\
&\leq 2E(\sup_s|X^n_s| |h^n_i-h_i|) + E(|Y^n_i-Y_i|)\\
&\leq 2E(\sup_s|X^n_s-X_s| |h^n_i-h_i|) + 2E(\sup_s|X_s| |h^n_i-h_i|) + 2E(\sup_s|X^n_s-X_s|)\\
&\leq 6E(\sup_s|X^n_s-X_s|) + 2E(\sup_s|X_s| |h^n_i-h_i|)
\end{align*}
Since $\sup_s|X_s| |h^n_i-h_i|$ converges a.s to zero and that for all $n$, $|h^n_i|\leq 1$, hence $\sup_s|X_s| |h^n_i-h_i|\leq 2\sup_s|X_s|$ and $\sup_s|X_s|\in L^1$, the Dominated Convergence Theorem implies that $E(\sup_s|X_s| |h^n_i-h_i|)\rightarrow 0$. Since by assumption $E(\sup_s|X^n_s-X_s|)\rightarrow 0$, it follows that $|E((H\cdot X)_T-(H^n\cdot X^n)_T)|$ converges to $0$. Letting $n$ tend to infinity in (\ref{eq:lim}) gives $|E((H\cdot X)_T)| \leq K$. So $X$ is a $\mathbb{G}$~quasimartingale, hence a $\mathbb{G}$~special semimartingale. Therefore $X$ has a $\mathbb G$ canonical decomposition $X=M+A$ where M is a $\mathbb G$~local martingale and $A$ is a $\mathbb G$~predictable finite variation process.

\textbf{Part(ii).} Let $(\tau_m)_{m\geq 1}$ be a sequence of bounded $\mathbb G$~stopping times that reduces $M$. Since for all $t$, $\mathcal{G}^m_t\stackrel{w}{\rightarrow}\mathcal{G}_t$, it follows from Lemma \ref{L:stop} that for each $m$ there exist a function $\phi_m$ strictly increasing and a sequence $(\tau_m^n)_{n\geq 1}$ such that $(\tau_m^{\phi_m(n)})_{n\geq 1}$ converges in probability to $\tau_m$ and $\tau_m^{\phi_m(n)}$ are bounded $\mathbb G^{\phi_m(n)}$ stopping times. We can extract a subsequence $(\tau_m^{\phi_m(\psi_m(n))})_{n\geq 1}$ converging a.s.~to $\tau_m$. In order to simplify the notation, fix $m\geq 1$ and up to working with $\tilde{\mathbb G}^n=\mathbb G^{\phi_m(\psi_m(n))}$ instead of $\mathbb G^n$ (which satisfies the same assumptions), take $\Phi_m := \phi_m \circ \psi_m$ to be the identity. Let $H$ be a $\mathbb G$~elementary predictable process as defined in Part (i). Since $\tau_m$ reduces $M$, $E((H\cdot A)_{\tau_m}) = E((H\cdot X)_{\tau_m})$. We can write
$$
E((H\cdot A)_{\tau_m})=E\big((H\cdot X)_{\tau_m}-(H^n\cdot X^n)_{\tau_m}\big) + E\big((H^n\cdot X^n)_{\tau_m}-(H^n\cdot X^n)_{\tau_m^n}\big) + E\big((H^n\cdot X^n)_{\tau_m^n}\big)
$$

We start with the third term. Since $H^n\cdot M^n$ is a $\mathbb G^n$~martingale and $\tau_m^n$ is a bounded $\mathbb G^n$~stopping time, it follows from Doob's optional sampling theorem that $E((H^n\cdot X^n)_{\tau_m^n}) = E((H^n\cdot A^n)_{\tau_m^n})$, hence $|E((H^n\cdot X^n)_{\tau_m^n})|\leq E(\int_0^{\tau_m}|dA^n_s|)\leq K$.

We focus now on the first term. Let $Y^i_s=X_s-X_{t_i}$ and $Y^{i,n}_s=X^n_s-X^n_{t_i}$.
\begin{align*}
E_1 &:= |E\big((H\cdot X)_{\tau_m}-(H^n\cdot X^n)_{\tau_m}\big)|\leq E\Big(\sup_{0\leq s\leq T}\big|(H\cdot X)_{s}-(H^n\cdot X^n)_{s}\big|\Big)\\
&\leq E\Big(\sum_{i=1}^k\sup_{t_i<s\leq t_{i+1}}\big|h_iY^i_s-h_i^nY^{i,n}_s\big|\Big)\\
&\leq \sum_{i=1}^k E\Big(\sup_{t_i<s\leq t_{i+1}} |Y^{i,n}_s||h^n_i-h_i| + \sup_{t_i<s\leq t_{i+1}} |h_i||Y^{i,n}_s-Y^i_s|\Big)
\end{align*}
Since $|h^i|\leq 1$, $|Y^{i,n}_s|\leq 2 \sup_{u}|X^n_u|$ and $|Y^{i,n}_s-Y^i_s|\leq 2 \sup_{u}|X^n_u-X_u|$, it follows that
\begin{align*}
E_1 &\leq 2 \sum_{i=1}^k \big\{ E\big(\sup_{0\leq u\leq T} |X^n_u||h^n_i-h_i|\big) + E\big(\sup_{0\leq u\leq T} |X^n_u-X_u|\big)\big\}\\
&\leq 6kE(\sup_{0\leq s\leq T}|X^n_s-X_s|)+2\sum_{i=1}^k\{E\big(\sup_{0\leq s\leq T}|X_s||h^n_i-h_i|\big)\}
\end{align*}

We study now the second term $E_2:=E\big((H^n\cdot X^n)_{\tau_m}-(H^n\cdot X^n)_{\tau_m^n}\big)$. Let $0<\eta<\min_i |t_{i+1}-t_i|$, and define $Y^n=H^n\cdot X^n$. Write now
$$
E\big(|Y^n_{\tau_m}-Y^n_{\tau^n_m}|\big)=E\big(|Y^n_{\tau_m}-Y^n_{\tau^n_m}|1_{\{|\tau_m-\tau_m^n|\leq \eta\}}\big)+E\big(|Y^n_{\tau_m}-Y^n_{\tau^n_m}|1_{\{|\tau_m-\tau_m^n|> \eta\}}\big)=: e_1 + e_2
$$
We study each of the two terms separately. We start with $e_2$.
\begin{align*}
e_2&\leq E\big((|Y^n_{\tau_m}|+|Y^n_{\tau^n_m}|)1_{\{|\tau_m-\tau_m^n|> \eta\}}\big)\leq 2 E\Big(\sum_{i=1}^k\sup_{t_i<s\leq t_{i+1}}|X^n_s-X^n_{t_i}|1_{\{|\tau_m-\tau_m^n|> \eta\}}\Big)\\
&\leq 4 k E\Big(\sup_{0\leq u\leq T}|X^n_u|1_{\{|\tau_m-\tau_m^n|> \eta\}}\Big)\\
&\leq 4 k E\Big(\sup_{0\leq u\leq T}|X^n_u-X_u|\Big)+4 k E\Big(\sup_{0\leq u\leq T}|X_u|1_{\{|\tau_m-\tau_m^n|> \eta\}}\Big)
\end{align*}
We study now $e_1$. On $\{|\tau_m-\tau^n_m|\leq \eta\}$ and since $\eta<\min_i|t_{i+1}-t_i|$, we have $|Y^n_{\tau_m}-Y^n_{\tau^n_m}|\leq 2\sup_{s\leq t\leq s+\eta}|X^n_t-X^n_s|$. In fact, one of the two following cases is possible for $\tau_m$ and $\tau_m^n$. Either they are both in the same interval $(t_i, t_{i+1}]$, in which case, $|Y^n_{\tau_m}-Y^n_{\tau^n_m}|=|h^n_i(X^n_{\tau_m}-X^n_{\tau_m^n})|\leq \sup_{s\leq t\leq s+\eta}|X^n_t-X^n_s|$, or they are in two consecutive intervals. For the second case, take for example $t_{i-1}<\tau_m\leq t_i<\tau_m^n\leq t_{i+1}$, then 
\begin{align*}
|Y^n_{\tau_m}-Y^n_{\tau^n_m}|&=|h^n_{i-1}(X^n_{\tau_m}-X^n_{t_{i-1}})-h^n_{i-1}(X^n_{t_i}-X^n_{t_{i-1}})-h^n_i(X^n_{\tau_m^n}-X^n_{t_i})|\\
&=|h^n_{i-1}(X^n_{\tau_m}-X^n_{t_i})-h^n_i(X^n_{\tau_m^n}-X^n_{t_i})|\leq |X^n_{\tau_m}-X^n_{t_i}|+|X^n_{\tau_m^n}-X^n_{t_i}|\\
&\leq 2\sup_{s\leq t\leq s+\eta}|X^n_t-X^n_s|
\end{align*}
The case $t_{i-1}<\tau_m^n\leq t_i<\tau_m\leq t_{i+1}$ is similar. Hence
$$
e_1\leq 2 E(\sup_{s\leq t\leq s+\eta}|X^n_t-X^n_s|1_{\{|\tau_m-\tau_m^n|\leq \eta\}})\leq 2 E(\sup_{s\leq t\leq s+\eta}|X^n_t-X^n_s|)
$$
Putting all this together yields for each $0<\eta<\min_{i}|t_{i+1}-t_i|$ and each $n$
\begin{align*}
|E(H\cdot A)_{\tau_m}|\leq K + &2 E(\sup_{s\leq t\leq s+\eta}|X^n_t-X^n_s|) + 2\sum_{i=1}^kE(\sup_{0\leq s\leq T}|X_s||h^n_i-h_i|)\\
&+ 4kE(\sup_{0\leq s\leq T}|X_s|1_{\{|\tau_m-\tau_m^n|>\eta\}}) + 10 k E\big(\sup_{0\leq u\leq T} |X^n_u-X_u|\big)
\end{align*}
Getting back to the general case, we obtain for each $m\geq 1$ and each $n\geq 1$,
\begin{align*}
|E(H\cdot A)_{\tau_m}|\leq K + &2 E(\sup_{s\leq t\leq s+\eta}|X^{\Phi_m(n)}_t-X^{\Phi_m(n)}_s|) + 2\sum_{i=1}^kE(\sup_{0\leq s\leq T}|X_s||h^{\Phi_m(n)}_i-h_i|)\\
&+ 4kE(\sup_{0\leq s\leq T}|X_s|1_{\{|\tau_m-\tau_m^{\Phi_m(n)}|>\eta\}}) + 10 k E\big(\sup_{0\leq u\leq T} |X^{\Phi_m(n)}_u-X_u|\big)
\end{align*}
As in the proof of Part (i), successive extractions allow us to find $\lambda_m(n)$ (independent from $i$) such that for all $1\leq i\leq k$, $h^{\lambda_m(n)}_i$ converges a.s.~to $h_i$. Letting $n$ go to infinity in
\begin{align*}
|E(H\cdot A)_{\tau_m}|\leq K + &2 E(\sup_{s\leq t\leq s+\eta}|X^{\lambda_m(n)}_t-X^{\lambda_m(n)}_s|) + 2\sum_{i=1}^kE(\sup_{0\leq s\leq T}|X_s||h^{\lambda_m(n)}_i-h_i|)\\
&+ 4kE(\sup_{0\leq s\leq T}|X_s|1_{\{|\tau_m-\tau_m^{\lambda_m(n)}|>\eta\}}) + 10 k E\big(\sup_{0\leq u\leq T} |X^{\lambda_m(n)}_u-X_u|\big)
\end{align*}
gives the estimate
$$
|E(H\cdot A)_{\tau_m}|\leq K + 2 \limsup_{n\rightarrow\infty}E(\sup_{s\leq t\leq s+\eta}|X^n_t-X^n_s|)
$$
Let $\lim_{\eta\rightarrow 0}\limsup_{n\rightarrow\infty}E(\sup_{s\leq t\leq s+\eta}|X^n_s-X^n_t|)=C$. Since $E(\sup_{s\leq t\leq s+\eta}|X^n_t-X^n_s|)\leq 2E(\sup_u|X^n_u|)\leq 2(E(\sup_u|M^n_u|+\int_0^T|dA^n_s|))\leq 4K$, $C<\infty$. Now letting $\eta$ go to zero yields finally $|E(H\cdot A)_{\tau_m}|\leq K+2C$, for each $m$. Thus $E(\int_0^{\tau_m}|dA_s|)\leq K+2C$, for each~$m$ and hence $E(\int_0^T|dA_s|)\leq K+2C$.

Now, $M=X-A=(X-X^n)+M^n+A^n-A$, and so
$$
\sup_{0\leq s\leq T}|M_s|\leq \sup_{0\leq s\leq T}|X_s-X^n_s|+\sup_{0\leq s\leq T}|M^n|+\int_0^T|dA^n_s|+\int_0^T|dA_s|
$$
Thus $E(\sup_{0\leq s\leq T}|M_s|)\leq 3K + 2C$ and $M$ is a $\mathbb G$~martingale.
\end{proof}

Once one obtains that $X$ is a $\mathbb G$~special semimartingale, one can be interested in characterizing the martingale $M$ and the finite variation predictable process $A$ in terms of the processes $M^n$ and $A^n$. M\'emin (Theorem 11 in \cite{Memin}) achieved this under ``extended convergence.'' Recall that $(X^n, \mathbb G^n)$ converges to $(X, \mathbb G)$ in the extended sense if for every $G\in \mathbb G_T$, the sequence of c\`adl\`ag processes $(X^n_t, E(1_G\mid\mathcal G^n_t))_{0\leq t\leq T}$ converges in probability under the Skorohod $J_1$ topology to $(X_t, E(1_G\mid\mathcal G_t))_{0\leq t\leq T}$. The author proves the following theorem. We refer to \cite{Memin} for a proof. In the theorem below, $\mathbb G^n$ and $\mathbb G$ are right-continuous filtrations.

\begin{theorem}\label{T:memin}
Let $(X^n)_{n\geq 1}$ be a sequence of $\mathbb G^n$~special semimartingales with canonical decompositions $X^n=M^n+A^n$ where $M^n$ is a $\mathbb G^n$~martingale and $A^n$ is a $\mathbb G^n$~predictable finite variation process. We suppose that the sequence $([X^n, X^n]_T^{\frac{1}{2}})_{n\geq 1}$ is uniformly integrable and that the sequence $(V(A^n)_T)_{n\geq 1}$ (where $V$ denotes the variation process) of real random variables is tight in $\mathbb R$. Let $X$ be a $\mathbb G$~quasi-left continuous special semimartingale with a canonical decomposition $X=M+A$ such that $([X, X]_T^{\frac{1}{2}})<\infty$.

If the extended convergence $(X^n, \mathbb G^n) \rightarrow (X, \mathbb G)$ holds, then $(X^n, M^n, A^n)$ converges in probability under the Skorohod $J_1$ topology to $(X, M, A)$.
\end{theorem}

In a filtration expansion setting, the sequence $X^n$ is constant and equal to some semimartingale $X$ of the base filtration. In this case the extended convergence assumption in Theorem \ref{T:memin} reduces to the weak convergence of the filtrations. We can deduce the following corollary from Theorems \ref{T:semimg2} and \ref{T:memin}.

\begin{corollary}\label{cor:semimg}
Let $(\mathbb G^n)_{n\geq 1}$ be a sequence of right-continuous filtrations and let $\mathbb G$ be a filtration such that $\mathcal G^n_t\stackrel{w}{\rightarrow}\mathcal G_t$ for all $t$. Let $X$ be a stochastic process such that for each $n$, $X$ is a $\mathbb G^n$~semimartingale with canonical decomposition $X = M^n + A^n$ such that there exists $K>0$, $E(\int_0^T|dA^n_s|) \leq K$ and $E(\sup_{0\leq s\leq T}|M^n_s|)\leq K$ for all $n$. Then
\begin{itemize}
	\item[(i)] If $X$ is $\mathbb G$~adapted, then $X$ is a $\mathbb G$~special semimartingale. 

	\item[(ii)] Assume moreover that $\mathbb G$ is right-continuous and let $X=M+A$ be the canonical decomposition of $X$. Then $M$ is a $\mathbb G$ martingale and $\sup_{0\leq s\leq T} |M_s|$ and $\int_0^T|dA_s|$ are integrable. 

	\item[(iii)] Furthermore, assume that $X$ is $\mathbb G$~quasi-left continuous and $\mathbb G^n \stackrel{w}{\rightarrow} \mathbb G$. Then $(M^n, A^n)$ converges in probability under the Skorohod $J_1$ topology to $(M,A)$.
\end{itemize}
\end{corollary}

\begin{proof}
The sequence $X^n=X$ clearly satisfies the assumptions of Theorem \ref{T:semimg2}, and the two first claims follow. For the last claim, notice that $[X,X]_T \in L^1$, so $\sqrt{[X,X]_T}\in L^1$ and hence $(\sqrt{[X,X]_t})_{0\leq t\leq T}$ is a uniformly integrable family of random variables. The tightness of the sequence of random variables $(V(A^n)_T)_{n\geq 1}$ follows from $E(\int_0^T|dA^n_s|) \leq K$ for any $n$ and some $K$ independent from~$n$.
\end{proof}

\subsection{Applications to filtration expansions}

We provide in this subsection a first application to the initial and progressive filtration expansions with a random variable and a general theorem on the progressive expansion with a process. We assume in the sequel that a right-continuous filtration $\mathbb F$ is given.

\subsubsection{Initial and progressive filtration expansions with a random variable}

Assume that $\mathbb F$ is the natural filtration of some c\`adl\`ag process. Let $\tau$ be a random variable and $\mathbb H$ and $\mathbb G$ the initial and progressive expansions of $\mathbb F$ with~$\tau$. In this subsection, the filtration $\mathbb G$ is considered only when $\tau$ is non negative. It is proved in \cite{KLP} that if $\tau$ satisfies Jacod's criterion i.e. if there exists a $\sigma$-finite measure $\eta$ on $\mathcal{B}(\mathbb R)$ such that 
$$
P(\tau \in \cdot\mid\mathcal F_t)(\omega) \ll \eta(.) \quad \text{a.s.}
$$
then every $\mathbb F$~semimartingale remains an $\mathbb H$ and $\mathbb G$~semimartingale. That it is an $\mathbb H$~semimartingale is due to Jacod \cite{Jacod:1987}. That it is also a $\mathbb G$~semimartingale  follows from Stricker's theorem. Its $\mathbb G$~decomposition is obtained in \cite{KLP} and this relies on the fact that these two filtrations \textit{coincide} after $\tau$. We provide now a similar but partial result for a random variable $\tau$ which may not satisfy Jacod's criterion. Assume there exists a sequence of random times $(\tau_n)_{n\geq 0}$ converging in probability to $\tau$ and let $\mathbb H^n$ and $\mathbb G^n$ be the initial and progressive expansions of $\mathbb F$ with $\tau_n$. The following holds.

\begin{theorem}\label{T:ini}
Let $M$ be an $\mathbb F$~martingale such that $\sup_{0 \leq t\leq T}|M_t|$ is integrable. Assume there exists an $\mathbb H^n$ predictable finite variation process $A^n$ such that $M-A^n$ is an $\mathbb H^n$ martingale. If there exists $K$ such that $E(\int_0^T|dA^n_s|)\leq K$ for all $n$, then $M$ is an $\mathbb H$ and $\mathbb G$ semimartingale. 
\end{theorem}

\begin{proof}
Since $\tau_n$ converges in probability to $\tau$ and $\mathbb F$ is the natural filtration of some c\`adl\`ag process, we can prove that $\mathcal H^n_t \stackrel{w}{\rightarrow} \mathcal H_t$ for each $t\in [0,T]$, using the same techniques as in Lemmas \ref{L:carac} and \ref{L:conv}. Up to replacing $K$ by $K+E(\sup_{0 \leq t\leq T}|M_t|)$, $M^n=M-A^n$ and $A^n$ satisfy the assumptions of Corollary~\ref{cor:semimg}. Therefore $M$ is an $\mathbb H$~semimartingale, and a $\mathbb G$~semimartingale by Stricker's theorem. 
\end{proof}

One case where the first assumption of Theorem \ref{T:ini} is satisfied is when $\tau_n$ satisfies Jacod's criterion, for each $n\geq 0$. In this case, and if $\mathbb F$ is the natural filtration of a Brownian motion $W$, the result above can be made more explicit. Assume for simplicity that the conditional distributions of $\tau_n$ are absolutely continuous w.r.t~Lebesgue measure,
$$
P(\tau_n\in du\mid\mathcal F_t)(\omega)=p^n_t(u,\omega)du
$$
where the conditional densities are chosen so that $(u,\omega, t)\rightarrow p^n_t(u,\omega)$ is c\`adl\`ag in $t$ and measurable for the optional $\sigma$-field associated with the filtration $\hat{\mathbb F}$ given by $\hat{\mathcal F}_t=\cap_{u>t}\mathcal B(\mathbb R)\otimes \mathcal F_u$. From the martingale representation theorem in a Brownian filtration, there exists for each $n$ a family $\{q^n(u), u>0\}$ of $\mathbb F$~predictable processes $(q^n_t(u))_{0\leq t\leq T}$ such that
\begin{equation}\label{eq:dens}
p^n_t(u)=p^n_0(u)+\int_0^tq^n_s(u)dW_s
\end{equation}

\begin{corollary}\label{cor:ini}
Assume there exists $K$ such that $E\Big(\int_0^T\Big|\frac{q^n_s(\tau_n)}{p^n_s(\tau_n)}\Big|ds\Big)\leq K$ for all $n$, then $W$ is a special semimartingale in both $\mathbb H$ and $\mathbb G$.
\end{corollary}

\begin{proof}
Since $\tau_n$ satisfy Jacod's criterion, it follows from Theorem 2.1 in \cite{Jacod:1987} that $W_t-A^n_t$ is an $\mathbb H$~local martingale, where $A^n_t=\int_0^t\frac{d\langle p^n(u), W\rangle_s}{p^n_s(u)}\Big|_{u=\tau_n}$. Now, it follows from (\ref{eq:dens}) that
$A^n_t=\int_0^t\frac{q^n_s(\tau_n)}{p^n_s(\tau_n)}ds$ and Theorem \ref{T:ini} allows us to conclude. 
\end{proof}

Assume the assumptions of Corollary \ref{cor:ini} are satisfied and let $W=M+A$ be the $\mathbb H$~canonical decomposition of $W$. Let $m$ be an $\mathbb F$~predictable process such that $\int_0^t m_s^2ds$ is locally integrable and let $M$ be the $\mathbb F$~local martingale $M_t=\int_0^tm_sdW_s$. Theorem VI.5 in \cite{Protter:2005} then guarantees that $M$ is an $\mathbb H$~semimartingale as soon as the process $(\int_0^tm_sdA_s)_{t\geq 0}$ exists as a path-by-path Lebesgue-Stieltjes integral a.s. See~\cite{JY} for a more comprehensive investigation of this result.

\begin{example}
In order to emphasize that some assumptions as in Theorem \ref{T:ini} are needed, we provide now a counter-example. Let $\mathbb F$ be the natural filtration of some Brownian motion~$B$ and choose $\tau$ to be some functional of the Brownian path i.e.~$\tau=f((B_s, 0\leq s\leq 1))$, such that $\sigma(\tau)=\mathcal F_1$. Then $B$ is not a semimartingale in $\mathbb H=(\mathcal F_t\vee\sigma(\tau))_{0\leq t\leq 1}$. Now, define $\tau_n=\tau+\frac{1}{\sqrt{n}}N$, where $N$ is a standard normal random variable independent from $\mathbb F$. Then $\tau_n$ converge a.s.~to $\tau$ and $P(\tau_n\leq u\mid\mathcal F_t)=\int_{-\infty}^uE(g_n(v-\tau)\mid\mathcal F_t)dv$ where $g_n$ is the probability density function of $\frac{1}{\sqrt{n}}N$, hence $P(\tau_n\in du\mid\mathcal F_t)(\omega)=E(g_n(u-\tau)\mid\mathcal F_t)(\omega)(du)$. Therefore, $\tau_n$ satisfies Jacod's criterion, for each $n$ and $p^n_t(u,\omega)=E(g_n(u-\tau)\mid\mathcal F_t)(\omega)$. Thus, $B$ is a semimartingale in $\mathbb H^n=(\mathcal F_t\vee\sigma(\tau_n)_{0\leq t\leq 1})$ and $\mathcal H^n_t\stackrel{w}{\rightarrow} \mathcal H_t$ for each $0\leq t\leq 1$.
\end{example}

\subsubsection{Progressive filtration expansion with a process}

Let $(N^n)_{n\geq 1}$ be a sequence of c\`adl\`ag processes converging in probability under the Skorohod $J_1$-topology to a c\`adl\`ag process $N$ and let $\mathbb{N}^n$ and $\mathbb{N}$ be their natural filtrations. Define the filtrations $\mathbb G^{0,n}=\mathbb{F}\vee\mathbb{N}^n$ and $\mathbb G^n$ by $\mathcal G_t^n=\bigcap_{u>t}\mathcal G^{0,n}_u$. Let also $\mathbb G^0$ (resp. $\mathbb G$) be the smallest (resp. the smallest right-continuous) filtration containing $\mathbb F$ and to which $N$ is adapted. The result below is the main theorem of this paper.

\begin{theorem}\label{T:exp}
Let $X$ be an $\mathbb{F}$~semimartingale such that for each $n$, $X$ is a $\mathbb{G}^n$~semimartingale with canonical decomposition $X = M^n + A^n$. Assume $E(\int_0^T|dA^n_s|) \leq K$ and $E(\sup_{0\leq s\leq T}|M^n_s|)\leq K$ for some $K$ and all $n$. Finally, assume one of the following holds.
\begin{itemize}
	\item[-] $N$ has no fixed times of discontinuity,
	
	\item[-] $N^n$ is a discretization of $N$ along some refining subdivision $(\pi_n)_{n\geq 1}$ such that each fixed time of discontinuity of $N$ belongs to $\cup_n \pi_n$.
\end{itemize}
Then 
\begin{itemize}
	\item[(i)] $X$ is a $\mathbb G^0$~special semimartingale. 

	\item[(ii)] Moreover, if $\mathbb F$ is the natural filtration of some c\`adl\`ag process then $X$ is a $\mathbb G$~special semimartingale with canonical decomposition $X=M+A$ such that $M$ is a $\mathbb{G}$ martingale and $\sup_{0\leq s\leq T} |M_s|$ and $\int_0^T|dA_s|$ are integrable.

	\item[(iii)] Furthermore, assume that $X$ is $\mathbb G$~quasi-left continuous and $\mathbb G^n \stackrel{w}{\rightarrow} \mathbb G$. Then $(M^n, A^n)$ converges in probability under the Skorohod $J_1$ topology to $(M,A)$.
\end{itemize}
\end{theorem}

\begin{proof}
Under assumption (i), since $N^n\stackrel{P}{\rightarrow}N$ and $P(\Delta N_t \neq 0)=0$ for all $t$, it follows from Lemma \ref{L:conv} that $\mathcal N^n_t\stackrel{P}{\rightarrow}\mathcal N_t$ for all $t$. The same holds under assumption (ii) using Lemma \ref{L:discr}. Lemma \ref{L:vee} then ensures that $\mathcal G^{0,n}_t\stackrel{w}{\rightarrow}\mathcal G^0_t$ for all $t$. Since $\mathcal G^{0,n}_t\subset\bigcap_{u>t}\mathcal G^{0,n}_u=\mathcal G^n_t$, it follows from Lemma \ref{L:subset} that $\mathcal G^n_t\stackrel{w}{\rightarrow}\mathcal G^0_t$ for all $t$. Being an $\mathbb{F}$~semimartingale, $X$ is clearly $\mathbb G^0$~adapted. An application of Corollary \ref{cor:semimg} ends the proof of the first claim. When $\mathbb F$ is the natural filtration of some c\`adl\`ag process, the same proofs as of Lemmas \ref{L:conv} and \ref{L:discr} guarantee that $\mathcal G^n_t\stackrel{w}{\rightarrow}\mathcal G_t$ for all $t$. Since $\mathbb G$ is right-continuous, the second and third claims follow from Corollary \ref{cor:semimg}.
\end{proof}

We apply this result to expand the filtration $\mathbb F$ progressively with a point process. Let $(\tau_i)_{i\geq 1}$ and $(X_i)_{i\geq 1}$ be two sequences of random variables such that for each $n$, the random vector $(\tau_1, X_1, \ldots, \tau_n, X_n)$ satisfies Jacod's criterion w.r.t the filtration $\mathbb F$. Assume that for all $t$ and $i$, $P(\tau_i=t)=0$ and that one of the following holds:
\begin{itemize}
	\item[(i)] For all $i$, $X_i$ and $\tau_i$ are independent, $E|X_i|=\mu$ for some $\mu$ and $\sum_{i=1}^{\infty}P(\tau_i\leq T)<\infty$
	
	\item[(ii)] $E(|X_i^2|)=c$ and $\sum_{i=1}^{\infty}\sqrt{P(\tau_i\leq T)}< \infty$.
\end{itemize}
Let $N^n_t=\sum_{i=1}^nX_i1_{\{\tau_i\leq t\}}$ and $N_t=\sum_{i=1}^{\infty}X_i1_{\{\tau_i\leq t\}}$. The assumptions on $N^n$ and $N$ as of Theorem \ref{T:exp} are satisfied.
\begin{lemma}
Under the assumptions above, $N_t\in L^1$ for each $t$, $N^n\stackrel{P}{\rightarrow} N$ and $N$ has no fixed times of discontinuity. 
\end{lemma}
\begin{proof}
We prove the statement under assumption (i). For each~$t$,
$$
E(|N_t|)\leq \sum_{i=1}^{\infty}E(|X_i|1_{\{\tau_i\leq t\}})\leq \mu\sum_{i=1}^{\infty}P(\tau_i\leq t)<\infty
$$
Therefore, $N_t\in L^1$. For $\eta>0$ and $n$ integer, we obtain the following estimate.
\begin{align*}
P(\sup_{0\leq t\leq T}|N_t-N^n_t|\geq\eta)&=P(\sup_{0\leq t\leq T}\big|\sum_{i=n+1}^{\infty}X_i1_{\{\tau_i\leq t\}}\big|\geq \eta)\\
&\leq P(\sup_{0\leq t\leq T}\sum_{i=n+1}^{\infty}|X_i|1_{\{\tau_i\leq t\}}\geq \eta)=P(\sum_{i=n+1}^{\infty}|X_i|1_{\{\tau_i\leq T\}}\geq \eta)\\
&\leq\frac{1}{\eta}E(\sum_{i=n+1}^{\infty}|X_i|1_{\{\tau_i\leq T\}})=\frac{\mu}{\eta}\sum_{i=n+1}^{\infty}P(\tau_i\leq T)\rightarrow 0
\end{align*}
This implies $N^n\stackrel{P}{\rightarrow} N$. Under assumption (ii), the proof is also straightforward and based on Cauchy Schwarz inequalities. Finally, since
$$
P(|\Delta N_t|\neq 0)\leq P(\exists i \mid \tau_i=t)\leq \sum_{i=1}^{\infty}P(\tau_i=t)=0
$$
$N$ has no fixed times of discontinuity
\end{proof}

Since the random vector $(\tau_1, X_1, \ldots, \tau_n, X_n)$ is assumed to satisfy Jacod's criterion, it follows from \cite{KLP} that $\mathbb F$~semimartingales remain $\mathbb G^n$~semimartingales, for each $n$. Therefore, this property also holds between $\mathbb F$ and $\mathbb G$ for $\mathbb F$~semimartingales whose $\mathbb G^n$~canonical decompositions satisfy the regularity assumptions of Theorem~\ref{T:exp}. Here $\mathbb G$ is the smallest filtration containing $\mathbb F$ and to which $N$ is adapted.

We would like to take a step further and reverse the previous situation. That is instead of starting with a sequence of processes $N^n$ converging to some process $N$, and putting assumptions on the semimartingale properties of $\mathbb F$~semimartingales w.r.t the intermediate filtrations~$\mathbb G^n$ and their decompositions therein, we would like to expand the filtration $\mathbb F$ with a given process $X$ and express all the assumptions in terms of $X$ and the $\mathbb F$~semimartingales considered. We are able to do this for c\`adl\`ag processes which satisfy a criterion that can loosely be seen as a localized extension of Jacod's criterion to processes. The integrability assumptions of Theorem \ref{T:exp} are expressed in terms of $\mathcal F_t$-conditional densities. Before doing this, we conclude this section by studying the stability of hypothesis $(H)$ with respect to the weak convergence of the $\sigma$-fields in a filtration expansion setting.

\subsection{The case of hypothesis $(H)$}\label{ss3.3}
Recall that given two nested filtrations $\mathbb F\subset\mathbb G$, we say that hypothesis $(H)$ holds between $\mathbb F$ and $\mathbb G$ if any square integrable $\mathbb F$ martingale remains a $\mathbb G$ martingale. Br\'emaud and Yor proved the next lemma (see \cite{BremaudYor}).
\begin{lemma}\label{L:H}
Let $\mathbb F\subset\mathbb G$ two nested filtrations. The following assertions are equivalent.
\begin{itemize}
\item[(i)] Hypothesis $(H)$ holds between $\mathbb F$ and $\mathbb G$.
 
\item[(ii)] For each $0\leq t\leq T$, $\mathcal F_T$ and $\mathcal G_t$ are conditionally independent given $\mathcal F_t$.

\item[(iii)] For each $0\leq t\leq T$, each $F\in L^2(\mathcal F_T)$ and each $G_t\in L^2(\mathcal G_t)$, 
$$
E(FG_t\mid\mathcal F_t)=E(F\mid\mathcal F_t)E(G_t\mid\mathcal F_t).
$$
\end{itemize}
\end{lemma}

Let $\mathbb F\subset\mathbb G$ be two nested right-continuous filtrations and $\mathbb G^n$ be a sequence of right-continuous filtrations containing $\mathbb F$ and such that $\mathcal G^n_t$ converges weakly to $\mathcal G_t$ for each $t$. We mentioned that an $\mathbb F$~local martingale that remains a $\mathbb G^n$~semimartingale for each $n$ might still lose its semimartingale property in $\mathbb G$ and we provided conditions that prevent this pathological behavior. In this subsection, we prove that this cannot happen in case hypothesis $(H)$ holds between $\mathbb F$ and each $\mathbb G^n$. One obtains even that hypothesis $(H)$ holds between $\mathbb F$ and $\mathbb G$.

\begin{theorem}
Let $\mathbb F$, $\mathbb G$ and $(\mathbb G^n)_{n\geq 1}$ right-continuous filtrations such that $\mathbb F\subset\mathbb G$, $\mathbb F\subset\mathbb G^n$ for each $n$ and $\mathcal G^n_t\stackrel{w}{\rightarrow}\mathcal G_t$ for each $t$. Assume that for each $n$, hypothesis $(H)$ holds between $\mathbb F$ and $\mathbb G^n$. Then hypothesis $(H)$ holds between $\mathbb F$ and $\mathbb G$.
\end{theorem}
\begin{proof}
We use Lemma \ref{L:H} and start with the bounded case. Let $0\leq t\leq T$, $F\in L^2(\mathcal F_T)$ and $G_t\in L^{\infty}(\mathcal G_t)$. For each $n$, define $G^n_t=E(G_t\mid\mathcal G^n_t)$. Then $G^n_t\in L^{\infty}(\mathcal G^n_t)$. Since hypothesis $(H)$ holds between $\mathbb F$ and $\mathbb G^n$, Lemma \ref{L:H} guarantees that $E(FG^n_t\mid\mathcal F_t)=E(F\mid\mathcal F_t)E(G^n_t\mid\mathcal F_t)$. But $\mathcal F_t\subset\mathcal G^n_t$, hence $E(G^n_t\mid\mathcal F_t)=E(E(G_t\mid\mathcal G^n_t)\mid\mathcal F_t)=E(G_t\mid\mathcal F_t)$. Since $\mathcal G^n_t\stackrel{w}{\rightarrow}\mathcal G_t$, $FG^n_t\stackrel{P}{\rightarrow}FG_t$. Now $FG^n_t$ is bounded by a square integrable process (by assumption) so the convergence holds in $L^1$ by the Dominated Convergence theorem so that $E(FG^n_t\mid\mathcal F_t)\stackrel{P}{\rightarrow}E(FG_t\mid\mathcal F_t)$. This proves that $E(FG_t\mid\mathcal F_t)=E(F\mid\mathcal F_t)E(G_t\mid\mathcal F_t)$. The general case where $G_t\in L^2(\mathcal G_t)$ follows by applying the bounded case result to the bounded random variables $G^{(m)}_t=G_t\wedge m$. Then for each $m$,
$$
E(FG^{(m)}_t\mid\mathcal F_t)\stackrel{P}{\rightarrow}E(FG^{(m)}_t\mid\mathcal F_t)
$$
and the Monotone Convergence theorem allows us to conclude.
\end{proof}

\section{A Filtration Expansion Result Based on an Assumption Involving Honest Times}\label{3bis}
This theorem is rather simple, but we include it both for completeness, and also because we need it later for our treatment of Bessel processes (see Section~\ref{Bessel}). 

Let $(\gep_n)_{n\geq 0}$ be a sequence of positive real numbers decreasing to zero.
We assume that the continuous adapted process $X$ is increasing to infinity and we define sequences of random times $(\tau^n_p)_{p\geq 0}$ to be
$$
\tau^n_p=\inf\{t\geq 0, X_t\geq p\gep_n\}
$$
Since $X$ is increasing to infinity, for $(\tau^n_p)_{p\geq 0}$ we have that for each $n\geq 1$, the sequence $(\tau^n_p)_{p\geq 1}$ is strictly increasing to infinity. That is, $\tau^n_p>\tau^n_{p-1}$ on the set where $\tau^n_{p-1}<\infty$, and $\lim_{p\rightarrow\infty}\tau^n_p=\infty$.
Define the sequence of processes
$$
X^n_t=\sum_{p=0}^{\infty}1_{\{\tau^n_p\leq t<\tau^n_{p+1}\}}X_{\tau^n_p}=\gep_n\sum_{p=0}^{\infty}p1_{\{\tau^n_p\leq t<\tau^n_{p+1}\}}
$$
and $\mathbb G^n$ (resp. $\mathbb G$) the progressive expansion of $\mathbb F$ with $X^n$ (resp. $X$). Let $\mathbb G^{\tau^n}$ be the smallest filtration containing $\mathbb F$ and that makes all $(\tau^n_p)_{p\geq 1}$ stopping times. We have the containment relation $\mathbb G^n\subset\mathbb G^{\tau^n}$. We make now the following assumption, noting that the validity of the assumption will depend on the process $X$ chosen.

\begin{assumption}[Honest times assumption]\label{A:6}
For each $n\geq 1$, the sequence $(\tau^n_p)_{p\geq 0}$ is an increasing sequence of $\mathbb F$~honest times such that $\tau^n_0=0$ and $\sup_p\tau^n_p=\infty$.
\end{assumption}

Under Assumption \ref{A:6}, the following holds (see Jeulin~\cite[Corollary 5.22]{Jeulin}).
\begin{theorem}[Jeulin]\label{T:4}
Let $M$ be an $\mathbb F$~local martingale. If Assumption \ref{A:6} holds, then $M-A^n$ is a $\mathbb G^{\tau^n}$ local martingale, where
\begin{equation}\label{eq:3}
A^n_t=\sum_{p=0}^{\infty}\int_0^t1_{\{\tau^n_p<s\leq \tau^n_{p+1}\}}\frac{1}{Z^{n,p+1}_{s^{-}}-Z^{n,p}_{s^{-}}}d\langle M, M^{n,p+1}-M^{n,p}\rangle_s
\end{equation}
and where $Z^{n,p}$ is the $\mathbb F$~optional projection of $\tau^n_p$ and $M^{n,p}$ is the martingale part in its Doob-Meyer decomposition.
\end{theorem}

Putting together Theorem \ref{T:4} above and Theorem~\ref{T:semimg2} we obtain the following result.

\begin{theorem}\label{T:5}
Suppose Assumption~\ref{A:6} holds. Let $M$ be an $\mathbb F$~martingale such that $\sup_{0\leq s\leq T} |M_s|$ is integrable. If $E(\int_0^T|dA^n_s|)\leq K$ for some $K$ and all $n\geq 1$, then $M$ is a $\mathbb G$~semimartingale. Here $A^n$ is defined in equation~(\ref{eq:3}).
\end{theorem}

\section{Filtration expansion with a c\`adl\`ag process satisfying a \textit{generalized Jacod's criterion} and applications to diffusions}\label{section4}

In this section, we assume a c\`adl\`ag process $X$ and a right-continuous filtration $\mathbb F$ are given. We assume throughout this section that our probability space is rich enough to contain non trivial continuous martingales. We study the case where the process $X$ and the filtration $\mathbb F$ satisfy the following assumption.

\begin{assumption}[Generalized Jacod's criterion]\label{A:J}
There exists a sequence $(\pi_n)_{n\geq 1}=(\{t^n_i\})_{n\geq 1}$ of subdivisions of $[0,T]$ whose mesh tends to zero and such that for each $n$, $(X_{t^n_0}, X_{t^n_1}-X_{t^n_0}, \ldots, X_{T}-X_{t^n_n})$ satisfies Jacod's criterion, i.e.~there exists a $\sigma$-finite measure $\eta_n$ on $\mathcal{B}(\mathbb{R}^{n+2})$ such that $P\big((X_{t^n_0}, X_{t^n_1}-X_{t^n_0}, \ldots, X_{T}-X_{t^n_n})\in \cdot\mid\mathcal F_t\big)(\omega) \ll \eta_n(\cdot)$~a.s.
\end{assumption}
Under Assumption \ref{A:J}, the $\mathcal F_t$-conditional density
$$
p^{(n)}_t(u_0,\ldots,u_{n+1},\omega)=\frac{P\big((X_{t^n_0}, X_{t^n_1}-X_{t^n_0}, \ldots, X_T-X_{t^n_n})\in (du_0,\ldots,du_{n+1})\mid\mathcal F_t\big)(\omega)}{\eta_n(du_0,\ldots,du_{n+1})}
$$
exists for each $n$, and can be chosen so that $(u_0,\ldots,u_{n+1}, \omega, t)\rightarrow p^{(n)}_t(u_0,\ldots,u_{n+1},\omega)$ is c\`adl\`ag in $t$ and measurable for the optional $\sigma$-field associated with the filtration $\hat{\mathbb{F}}_t$ given by $\hat{\mathcal{F}}_t=\cap_{u>t}\mathcal{B}(\mathbb{R}^{n+2})\otimes \mathcal{F}_u$.
For each $0\leq i\leq n$, define 
$$
p^{i,n}_t(u_0,\ldots,u_i)=\int_{\mathbb{R}^{n+1-i}}p^{(n)}_t(u_0,\ldots,u_{n+1})\eta_n(du_{i+1},\ldots,du_{n+1})
$$
Let $M$ be a continuous $\mathbb F$~local martingale. Define
\begin{equation}\label{eq:Ain}
A^{i,n}_t=\int_{0}^t\frac{d\langle p^{i,n}(u_0,\ldots,u_i),M\rangle_s}{p^{i,n}_{s^{-}}(u_0,\ldots,u_i)}\Big|_{\forall 0\leq k\leq i, u_k=X_{t^n_k}-X_{t^n_{k-1}}}
\end{equation}
Finally define
$$
A^{(n)}_t=\sum_{i=0}^n\int_{t\wedge t^n_i}^{t\wedge t^n_{i+1}}dA^{i,n}_s
$$
i.e.
\begin{equation}\label{eq:An}
A^{(n)}_t=\sum_{i=0}^n1_{\{t^n_i\leq t<t^n_{i+1}\}}\big(\sum_{k=0}^{i-1}\int_{t^n_k}^{t^n_{k+1}}dA^{k,n}_s+\int_{t^n_i}^tdA^{i,n}_s\big)
\end{equation}
Of course, on each time interval $\{t^n_i\leq t<t^n_{i+1}\}$, only one term appears in the outer sum. Let $\mathbb G^0$ (resp. $\mathbb G$) be the smallest (resp. the smallest right-continuous) filtration containing $\mathbb F$ and relative to which $X$ is adapted. The theorem below is the main result of this section.

\begin{theorem}\label{T:semimgX}
Assume $X$ and $\mathbb F$ satisfy Assumption \ref{A:J} and that one of the following holds.
\begin{itemize}
	\item[-] $X$ has no fixed times of discontinuity,
	
	\item[-] the sequence of subdivisions $(\pi_n)_{n\geq 1}$ in Assumption \ref{A:J} is refining and each fixed time of discontinuity of $X$ belongs to $\cup_n \pi_n$.
\end{itemize}
Let $M$ be a continuous $\mathbb F$~martingale such that $E(\sup_{s\leq T}|M_s|)\leq K$ and $E(\int_0^T|dA^{(n)}_s|)\leq K$ for some $K$ and all $n$, with $A^n$ as in (\ref{eq:An}). Then
\begin{itemize}
	\item[(i)] $M$ is a $\mathbb G^0$~special semimartingale. 

	\item[(ii)] Moreover, if $\mathbb F$ is the natural filtration of some c\`adl\`ag process $Z$, then $M$ is a $\mathbb G$~special semimartingale with canonical decomposition $M=N+A$ such that $N$ is a $\mathbb G$ martingale and $\sup_{0\leq s\leq T} |N_s|$ and $\int_0^T|dA_s|$ are integrable.
\end{itemize}
\end{theorem}

\begin{proof}
We construct the discretized process $X^n$ defined by $X^n_t=X_{t^n_k}$ for all $t^n_k\leq t<t^n_{k+1}$. That is 
$$
X^n_t=\sum_{i=0}^nX_{t^n_i}1_{\{t^n_i\leq t<t^n_{i+1}\}} + X_T1_{\{t=T\}}
$$
with the convention $t^n_0=0$ and $t^n_{n+1}=T$. Let $\mathbb G^n$ be the smallest right-continuous filtration containing $\mathbb F$ and to which $X^n$ is adapted.

Now, for $0\leq t\leq T$,
\begin{align*}
X^n_t&=\sum_{i=0}^nX_{t^n_i}1_{\{t^n_i\leq t<t^n_{i+1}\}}+X_T1_{\{t=T\}}=\sum_{i=0}^nX_{t^n_i}1_{\{t^n_i\leq t\}}-\sum_{i=0}^nX_{t^n_i}1_{\{t^n_{i+1}\leq t\}}+X_T1_{\{t=T\}}\\
&=\sum_{i=1}^n(X_{t^n_i}-X_{t^n_{i-1}})1_{\{t^n_i\leq t\}}+X_{0}1_{\{t^n_0\leq t\}}-X_{t^n_{n}}1_{\{t^n_{n+1}\leq t\}}+X_T1_{\{t^n_{n+1}\leq t\}}\\
&=X_{0}1_{\{t^n_0\leq t\}}+\sum_{i=1}^{n+1}(X_{t^n_i}-X_{t^n_{i-1}})1_{\{t^n_i\leq t\}}=\sum_{i=0}^{n+1}(X_{t^n_i}-X_{t^n_{i-1}})1_{\{t^n_i\leq t\}}
\end{align*}
with the notation $X_{t^n_{-1}}=0$.

For each $0\leq i\leq n+1$, let $\mathbb H^{i,n}$ be the initial expansion of $\mathbb F$ with $(X_{t^n_k}-X_{t^n_{k-1}})_{0\leq k\leq i}$. Since $(X_{t^n_k}-X_{t^n_{k-1}})_{0\leq k\leq i}$ satisfies Jacod's criterion, it follows that for each $0\leq i\leq n+1$, $M-A^{i,n}$ is an $\mathbb H^{i,n}$~local martingale. Let 
$$
\tilde{\mathcal G}^n_t=\bigcap_{u>t}\mathcal{F}_u\vee\sigma\big((X_{t^n_i}-X_{t^n_{i-1}})1_{\{t^n_i\leq u\}}, i=0,\ldots,n+1\big)
$$
Since the times $t^n_k$ are fixed, $\mathbb H^{i,n}$ is also the initial expansion of $\mathbb F$ with $(t^n_k, X_{t^n_k}-X_{t^n_{k-1}})_{0\leq k\leq i}$ and $\tilde{\mathbb G}^n=\mathbb G^n$ using a Monotone Class argument and the fact that $X^n_{t^n_k}=X_{t^n_k}$, for all $0\leq k\leq n+1$. So it follows from Theorem 8 in \cite{KLP} that $M-A^{(n)}$ is a $\mathbb G^n$~local martingale. An application of Theorem \ref{T:exp} yields the result.
\end{proof}

We refrain from stating Theorem \ref{T:semimgX} in a more general form for clarity but provide two extensions in the remarks below.
\begin{itemize}

	\item[(i)] Going beyond the continuous case for the $\mathbb F$ local martingale $M$  is straightforward. We only need to use Theorem 8 in \cite{KLP} in its general version rather than its application to the continuous case. However the explicit form of $A^{(n)}$ is much more complicated, which makes it hard to check the integrability assumption of Theorem~\ref{T:semimgX}. To be more concrete, one has to replace $A^{(n)}$ in the theorem above by $\tilde{A}^{(n)}$ defined by
$$
\tilde{A}^{(n)}_t=\sum_{i=0}^n\int_{t\wedge t^n_i}^{t\wedge t^n_{i+1}}(d\tilde{A}^{i,n}_s+dJ^{i,n}_s)
$$
i.e.
$$
\tilde{A}^{(n)}_t=\sum_{i=0}^n1_{\{t^n_i\leq t<t^n_{i+1}\}}\Big(\sum_{k=0}^{i-1}\int_{t^n_k}^{t^n_{k+1}}(d\tilde{A}^{k,n}_s+dJ^{k,n}_s)+\int_{t^n_i}^t(d\tilde{A}^{i,n}_s+dJ^{i,n}_s)\Big)
$$
where $\tilde{A}^{i,n}$ is the compensator of $M$ in $\mathbb{H}^{i,n}$ as given by Jacod's theorem (see Theorems VI.10 and VI.11 in \cite{Protter:2005}) and $J^{i,n}$ is the dual predictable projection of $\Delta M_{t^n_{i+1}}1_{[t^n_{i+1},\infty[}$ onto~$\mathbb H^{i,n}$.

	\item[(ii)] A careful study of the proof above shows that Assumption \ref{A:J} is only used to ensure that there exists an $\mathbb H^{i,n}$ predictable process $A^{i,n}$ such that $M-A^{i,n}$ is an $\mathbb H^{i,n}$~local martingale. Therefore, Theorem \ref{T:semimgX} will hold whenever this weaker assumption is satisfied.
\end{itemize}

If the sequence of filtrations $\mathbb G^n$ converges weakly to $\mathbb G$ then $(M-A^{(n)}, A^{(n)})$ converges in probability under the Skorohod $J_1$ topology to $(N, A)$. Many criteria for this to hold are provided in the literature, see for instance Propositions $3$ and $4$ in \cite{Coquet0}. This holds for example when every $\mathbb G$~martingale is continuous and the subdivision $(\pi_n)_{n\geq 1}$ is refining. In this case, for each $0\leq t\leq T$, $(\mathcal G^n_t)_{n\geq 1}$ is increasing and converges weakly to the $\sigma$-field~$\mathcal G_t$. The following lemma allows us to conclude. See \cite{Coquet0} for a proof. 

\begin{lemma}\label{L:conv2}
Assume that every $\mathbb G$~martingale is continuous and that for every $0\leq t\leq T$, $(\mathcal G^n_t)_{n\geq 1}$ increases (or decreases) and converges weakly to $\mathcal G_t$. Then $\mathbb G^n\stackrel{w}{\rightarrow}\mathbb G$.
\end{lemma}


\subsection{Application to diffusions}

Start with a Brownian filtration $\mathbb F=(\mathcal F_t)_{0\leq t\leq T}$, $\mathcal F_t=\sigma(B_s, s\leq t)$ and consider the stochastic differential equation
$$
dX_t=\sigma(X_t)dB_t+b(X_t)dt
$$
Assume the existence of a unique strong solution $(X_t)_{0\leq t\leq T}$. Assume in addition that the transition density $\pi(t,x,y)$ exists and is twice continuously differentiable in $x$ and continuous in $t$ and $y$. This is guaranteed for example if $b$ and $\sigma$ are infinitely differentiable with bounded derivatives and if the H\"{o}rmander condition holds for any~$x$ (see \cite{BallyTalay}), and we assume that this holds in the sequel. In this case, $\pi$ is even infinitely differentiable. 

We next show how we can expand a filtration dynamically as $t$ increases, via another stochastic process evolving backwards in time.  To this end, define the time reversed process $Z_t=X_{T-t}$, for all $0\leq t\leq T$. Let $\mathbb G=(\mathcal G_t)_{0\leq t<\frac{T}{2}}$ be the smallest right-continuous filtration containing $(\mathcal F_t)_{0\leq t<\frac{T}{2}}$ and to which $(Z_t)_{0\leq t<\frac{T}{2}}$ is adapted.
 We would like to prove that $B$ remains a special semimartingale in $\mathbb G$ and give its canonical decomposition. That $B$ is a $\mathbb G$ semimartingale can be obtained using the usual results from the filtration expansion theory. However, our approach allows us to obtain the decomposition, too. We assume (w.l.o.g) that $T=1$. Introduce the reversed Brownian motion $\tilde{B}_t=B_{1-t}-B_1$ and the filtration $\tilde{\mathbb G}=(\tilde{\mathcal{G}_t})_{0\leq t<\frac{1}{2}}$ defined by
$$
\tilde{\mathcal G}_t=\bigcap_{t<u<\frac{1}{2}}\sigma(B_s, \tilde{B}_s, 0\leq s<u).
$$

\begin{theorem}
Both $B$ and $\tilde{B}$ are $\mathbb G$~semimartingales.
\end{theorem}

\begin{proof}
First, it is well known that $\tilde{B}$ is a Brownian motion in its own natural filtration and $\sigma(B_{1-s}-B_1, 0\leq s<\frac{1}{2})$ is independent from $\sigma(B_s, 0\leq s<\frac{1}{2})$. Therefore $(B_t)_{0\leq t<\frac{1}{2}}$ and $(\tilde{B}_t)_{0\leq t<\frac{1}{2}}$ are independent Brownian motions in $\tilde{\mathbb G}$. Now, given our strong assumptions on the coefficients $b$ and $\sigma$, $X_1$ satisfies Jacod's criterion with respect to $\tilde{\mathbb G}$. Therefore $B$ and $\tilde{B}$ remain semimartingales in $\mathbb H=(\mathcal H_t)_{0\leq t<\frac{1}{2}}$ where $\mathcal H_t=\bigcap_{\frac{1}{2}>u>t}\tilde{\mathcal G}_u\vee \sigma(X_1)$. It only remains to prove that $\mathbb G=\mathbb H$. For this, use Theorem~$V.23$ in \cite{Protter:2005} to get that 
$$
dX_{1-t}=\sigma(X_{1-t})d\tilde{B}_t+(\sigma^{'}(X_{1-t})\sigma(X_{1-t})+b(X_{1-t}))dt
$$
Since $b+\sigma\sigma^{'}$ and $\sigma$ are Lipschitz, $\bigcap_{\frac{1}{2}>u>t}\sigma(X_{1-s}, 0\leq s\leq u)=\bigcap_{\frac{1}{2}>u>t}\sigma(\tilde{B}_s, 0\leq s\leq u)\vee \sigma(X_1)$ and the result follows.
\end{proof}

We apply now our results to obtain the $\mathbb G$~decomposition. 
This is the primary result of this article.

\begin{theorem}\label{T:diff}
Assume there exists a nonnegative function $\phi$ such that $\int_0^1\phi(s)ds <\infty$ and for each $0\leq s < t$,
$$
E\Big(\Big|\frac{1}{\pi}\frac{\partial \pi}{\partial x}(t-s, X_s, X_t)\Big|\Big)\leq \phi(t-s)
$$
Then the process $(B_t)_{0\leq t<\frac{1}{2}}$ is a $\mathbb G$ semimartingale and 
$$
B_t - \int_0^t\frac{1}{\pi}\frac{\partial \pi}{\partial x}(1-2s, X_s, X_{1-s})ds
$$
is a $\mathbb G$~Brownian motion.
\end{theorem}

\begin{proof}
Since the process $Z_t$ is a c\`adl\`ag process with no fixed times of discontinuity, we can apply Theorem \ref{T:semimgX}. First we prove that $(Z_t)_{0\leq t<\frac{1}{2}}$ and $(\mathcal F_t)_{0\leq t<\frac{1}{2}}$ satisfy Assumption~\ref{A:J}. Let $(\pi_n)_{n\geq 1}=(\{t^n_i\})_{n\geq 1}$ be a refining  sequence of subdivisions of $[0,\frac{1}{2}]$ whose mesh tends to zero. We will do more and compute directly the conditional distributions of $(Z_{t^n_0}, Z_{t^n_1}-Z_{t^n_0}, \ldots, Z_{t^n_i}-Z_{t^n_{i-1}})$ for any $1\leq i\leq n+1$. Pick such $i$ and let $0\leq t<\frac{1}{2}$ and $(z_0,\ldots, z_i)\in~\mathbb R^{i+1}$.
\begin{align*}
P(Z_{t^n_0}&\leq z_0, Z_{t^n_1}-Z_{t^n_0}\leq z_1, \ldots, Z_{t^n_i}-Z_{t^n_{i-1}}\leq z_i\mid\mathcal F_t)\\
&=P(X_1\leq z_0, X_1-X_{1-t^n_1}> -z_1, \ldots, X_{1-t^n_{i-1}}-X_{1-t^n_i}> -z_i\mid\mathcal F_t)\\
&=E\Big(\prod_{k=1}^{i-1}1_{\{X_{1-t^n_k}-X_{1-t^n_{k+1}}\geq -z_{k+1}\}} P\big(X_{1-t^n_1}-z_1\leq X_1\leq z_0\mid \mathcal F_{1-t^n_1}\big)\mid \mathcal F_t\Big)\\
&=E\Big(\prod_{k=1}^{i-1}1_{\{X_{1-t^n_k}-X_{1-t^n_{k+1}}\geq -z_{k+1}\}} \int_{X_{1-t^n_1}-z_1}^{\infty} 1_{\{u_1\leq z_0\}}P_{X_{1-t^n_1}}(t^n_1,u_1)du_1\mid \mathcal F_t\Big)\\
&=E\Big(\prod_{k=1}^{i-1}1_{\{X_{1-t^n_k}-X_{1-t^n_{k+1}}\geq -z_{k+1}\}} \int_{-z_1}^{\infty} 1_{\{v_1\leq z_0-X_{1-t^n_1}\}}P_{X_{1-t^n_1}}(t^n_1,v_1+X_{1-t^n_1})dv_1\mid \mathcal F_t\Big)
\end{align*}
Repeating the same technique and conditioning successively w.r.t $\mathcal F_{1-t^n_2}, \ldots, \mathcal F_{1-t^n_i}$ gives
\begin{align*}
P(Z_{t^n_0}&\leq z_0, Z_{t^n_1}-Z_{t^n_0}\leq z_1, \ldots, Z_{t^n_i}-Z_{t^n_{i-1}}\leq z_i\mid\mathcal F_t)=E\Big(\int_{-z_i}^{\infty}\cdots\int_{-z_1}^{\infty}\\
&1_{\{\sum_{k=1}^iv_k\leq z_0-X_{1-t^n_i}\}}\prod_{k=1}^iP_{X_{1-t^n_i}+\sum_{j=k+1}^iv_j}(t^n_k-t^n_{k-1},\sum_{l=k}^iv_l+X_{1-t^n_i})dv_1\ldots dv_i\mid \mathcal F_t\Big)\\
&=\int_{-\infty}^{\infty}\int_{-z_i}^{\infty}\cdots\int_{-z_1}^{\infty}1_{\{u+\sum_{k=1}^iv_k\leq z_0\}}P_{X_t}(1-t^n_i-t,u)\\
&\qquad \qquad \qquad \qquad\prod_{k=1}^iP_{u+\sum_{j=k+1}^iv_j}(t^n_k-t^n_{k-1},u+\sum_{l=k}^iv_l)dv_1\ldots dv_idu
\end{align*}
Fubini's Theorem implies then
\begin{align*}
P(Z_{t^n_0}&\leq z_0, Z_{t^n_1}-Z_{t^n_0}\leq z_1, \ldots, Z_{t^n_i}-Z_{t^n_{i-1}}\leq z_i\mid\mathcal F_t)=\int_{-z_i}^{\infty}\cdots\int_{-z_1}^{\infty}\int_{-\infty}^{z_0-\sum_{k=1}^iv_k}\\
&P_{X_t}(1-t^n_i-t,u)\prod_{k=1}^iP_{u+\sum_{j=k+1}^iv_j}(t^n_k-t^n_{k-1},u+\sum_{l=k}^iv_l)dudv_1\ldots dv_i
\end{align*}
Since the transition density $\pi(t,x,y)=P_x(t,y)$ is twice continuously differentiable in $x$ by assumption, it is straightforward to check that 
$$
p^{i,n}_t(z_0,\ldots,z_i)=\prod_{k=1}^i\pi(t^n_k-t^n_{k-1}, \sum_{j=0}^kz_j, \sum_{j=0}^{k-1}z_j) \pi(1-t^n_i-t, X_t, \sum_{j=0}^iz_j)
$$
One then readily obtains
$$
d\langle p^{i,n}_{.}(z_0, \ldots, z_i), B_{.}\rangle _s = \frac{1}{\pi}\frac{\partial \pi}{\partial x}(1-t^n_i-s, X_s, \sum_{j=0}^iz_k)p^{i,n}_s(z_0, \ldots, z_i)ds
$$
Hence by taking the local martingale $M$ in (\ref{eq:Ain}) to be $B$, we get
$$
A^{i,n}_t=\int_0^t\frac{1}{\pi}\frac{\partial \pi}{\partial x}(1-t^n_i-s, X_s, X_{1-t^n_i})ds
$$
Now equation (\ref{eq:An}) becomes
\begin{align*}
A^{(n)}_t=\sum_{i=0}^n1_{\{t^n_i\leq t<t^n_{i+1}\}}\Big(&\sum_{k=0}^{i-1}\int_{t^n_k}^{t^n_{k+1}}\frac{1}{\pi}\frac{\partial \pi}{\partial x}(1-t^n_k-s, X_s, X_{1-t^n_k})ds\\
&+\int_{t^n_i}^t\frac{1}{\pi}\frac{\partial \pi}{\partial x}(1-t^n_i-s, X_s, X_{1-t^n_i})ds\Big)
\end{align*}
In order to apply Theorem \ref{T:semimgX}, it only remains to prove that $E(\int_0^{\frac{1}{2}}|dA^{(n)}_s|)\leq K$ for some constant $K$ independent from $n$. The finite constant $K=\int_0^1\phi(s)ds$ works since
\begin{align*}
E\big(\int_0^{\frac{1}{2}}|dA^{(n)}_s|\big)&\leq\sum_{k=0}^n\int_{t^n_k}^{t^n_{k+1}}E\Big|\frac{1}{\pi}\frac{\partial \pi}{\partial x}(1-t^n_k-s, X_s, X_{1-t^n_k})\Big|ds\\
&\leq \sum_{k=0}^n\int_{t^n_k}^{t^n_{k+1}}\phi(1-t^n_k-s)ds = \sum_{k=0}^n\int_{1-t^n_k-t^n_{k+1}}^{1-2t^n_{k}}\phi(s)ds\\
&\leq \sum_{k=0}^n\int_{1-2t^n_{k+1}}^{1-2t^n_{k}}\phi(s)ds = \int_0^1\phi(s)ds
\end{align*}
This proves again that $B$ is a $\mathbb G$~semimartingale. Now $A^{(n)}$ converges in probability to the process $A$ given by
$$
A_t=\int_0^t\frac{1}{\pi}\frac{\partial \pi}{\partial x}(1-2s, X_s, X_{1-s})ds
$$
Since all $\mathbb G$~martingales are continuous, the comment following Theorem \ref{T:semimgX} ensures that B-A is a $\mathbb G$~martingale. Its quadratic variation is $t$, therefore it is a $\mathbb G$~Brownian motion.
\end{proof}

In the Brownian case, the result in Theorem \ref{T:diff} can also be obtained using the usual theory of initial expansion of filtration. Assume $b=0$ and $\sigma=1$, i.e.~$Z=B_{1-\cdot}$ and $X=B$.

\begin{theorem}\label{T:Brownian}
The process $B$ is a $\mathbb G$~semimartingale and
$$
B_t - \int_0^t\frac{B_{1-s}-B_s}{1-2 s}ds, \qquad 0\leq t<\frac{1}{2}
$$
is a $\mathbb G$~Brownian motion.
\end{theorem}

\begin{proof}
Introduce the filtration $\mathbb H^1=(\mathcal H_t)_{0\leq t<\frac{1}{2}}$ obtained by initially expanding $\mathbb F$ with~$B_{\frac{1}{2}}$. 
$$
\mathcal H^1_t = \bigcap_{u>t} F_u \vee \sigma(B_{\frac{1}{2}})
$$
We know that $B$ remains an $\mathbb H^1$~semimartingale and 
$$
M_t:=B_t-\int_0^t\frac{B_{\frac{1}{2}}-B_s}{\frac{1}{2}-s}ds, \qquad 0\leq t<\frac{1}{2}
$$
is an $\mathbb H^1$~Brownian motion. Now expand initially $\mathbb H^1$ with the independent $\sigma$-field $\sigma (B_v - B_{\frac{1}{2}}, \frac{1}{2}<v\leq 1)$ to obtain $\mathbb H$ i.e.
$$
\mathcal H_t=\bigcap_{u>t} H^1_u\vee \sigma(B_v - B_{\frac{1}{2}}, \frac{1}{2}<v\leq 1)
$$
Obviously $(M_t)_{0\leq t<\frac{1}{2}}$ remains an $\mathbb H$~Brownian motion. But $\mathcal G_t \subset \mathcal H_t$, for all $0\leq t<\frac{1}{2}$, hence the optional projection of $M$ onto $\mathbb G$, denoted $^oM$ in the sequel, is again a martingale (see \cite{ProtterFollmer}), i.e.
$$
^oM_t=B_t-E(\int_0^t\frac{B_{\frac{1}{2}}-B_s}{\frac{1}{2}-s}ds\mid\mathcal G_t), \qquad 0\leq t<\frac{1}{2}
$$
is a $\mathbb G$~martingale. Also, $N_t:=E(\int_0^t\frac{B_{\frac{1}{2}}-B_s}{\frac{1}{2}-s}ds\mid\mathcal G_t)-\int_0^tE(\frac{B_{\frac{1}{2}}-B_s}{\frac{1}{2}-s}\mid\mathcal G_s)ds$ is a $\mathbb G$~local martingale, see for example \cite{KLP} for a proof. So
$$
B_t\ =\ ^oM_t+N_t+\int_0^tE(\frac{B_{\frac{1}{2}}-B_s}{\frac{1}{2}-s}\mid\mathcal G_s)ds
$$
We prove now the theorem using properties of the Brownian bridge. Recall that for any $0\leq T_0 <T_1<\infty$,
\begin{equation}\label{eq:condLaw}
\mathcal L\Big( (B_t)_{T_0\leq t\leq T_1} \mid B_s, s\notin ]T_0, T_1[\Big)=\mathcal L\Big( (B_t)_{T_0\leq t\leq T_1} \mid B_{T_0}, B_{T_1} \Big)
\end{equation}
and
$$
\mathcal L\Big( (B_t)_{T_0\leq t\leq T_1} \mid B_{T_0}=x, B_{T_1}=y \Big)=\mathcal L\Big(x+\frac{t-T_0}{T_1-T_0}(y-x)+(Y^{W,T_1-T_0}_{t-T_0})_{T_0\leq t\leq T_1}\Big)
$$
where $W$ is a generic standard Brownian motion and $Y^{W, T_1-T_0}$ is the standard Brownian bridge on $[0, T_1-T_0]$. It follows that for all $ T_0\leq t\leq T_1$ and all $x$ and $y$,
\begin{equation}\label{eq:bridge}
E(B_t \mid B_{T_0} = x, B_{T_1} = y)=\frac{T_1-t}{T_1-T_0}x+\frac{t-T_0}{T_1-T_0}y
\end{equation}
For any $0\leq s< t <\frac{1}{2}$, if follows from (\ref{eq:condLaw}) and (\ref{eq:bridge}) that
$$
E(B_{\frac{1}{2}}-B_s\mid\mathcal G_s)=\frac{1}{2}(B_{1-s}-B_s)
$$
Therefore
$$
B_t-\int_0^t\frac{B_{1-s}-B_s}{1-2s}ds\ =\ ^oM_t+N_t
$$
is a $\mathbb G$~local martingale. Since the quadratic variation of the $\mathbb G$~local martingale $B-A$ is~$t$, Levy's characterization of Brownian motion ends the proof.
\end{proof}

In the immediately previous proof, the properties of the Brownian bridge allow us to compute explicitly the decomposition of $B$ in $\mathbb G$. Our method obtains both the semimartingale property and the decomposition simultaneously and generalizes to diffusions, for which the computations as in the proof of Theorem \ref{T:Brownian} are hard. We provide a shorter proof for Theorem \ref{T:Brownian} based on Theorem \ref{T:diff}.  This illustrated that, given Theorem~\ref{T:diff}, even in the Brownian case our method is shorter, simpler, and more intuitive.

\begin{proof}\textbf{[Second proof of Theorem \ref{T:Brownian}]} In the Brownian case, $\pi(t,x,y)=\frac{1}{\sqrt{2\pi t}}e^{-\frac{(y-x)^2}{2t}}$. Therefore $\frac{1}{\pi}\frac{\partial \pi}{\partial x}(t,x,y)=\frac{y-x}{t}$. Hence
$$
E\Big(\big|\frac{1}{\pi}\frac{\partial \pi}{\partial x}\big|(t-s, B_s, B_t)\Big)\leq \frac{1}{t-s}E(|B_t-B_s|)=\sqrt{\frac{2}{\pi}}\frac{1}{\sqrt{t-s}}
$$
and $\phi(x)=\sqrt{\frac{2}{\pi}}\frac{1}{\sqrt{x}}$ is integrable in zero. From the closed formula for the transition density, $A_t=\int_0^t\frac{B_{1-s}-B_s}{1-2 s}ds$. Therefore $B$ is a $\mathbb G$~semimartingale, and $B-A$ is a $\mathbb G$~Brownian motion by Theorem~\ref{T:diff}.
\end{proof}

This property satisfied by Brownian motion is inherited by diffusions whose parameters $b$ and $\sigma$ satisfy some boundedness assumptions. We add the extra assumptions that $b$ and $\sigma$ are bounded and $k\leq\sigma(x)$ for some $k>0$. The following holds.

\begin{corollary}
The process $(B_t)_{0\leq t<\frac{1}{2}}$ is a $\mathbb G$ semimartingale and 
$$
B_t - \int_0^t\frac{1}{\pi}\frac{\partial \pi}{\partial x}(1-2s, X_s, X_{1-s})ds
$$
is a $\mathbb G$~Brownian motion.
\end{corollary}

\begin{proof}
Introduce the following quantities
$$
s(x)=\int_0^x\frac{1}{\sigma(y)}dy \qquad g=s^{-1} \qquad \mu=\frac{b}{\sigma}\circ g - \frac{1}{2}\sigma^{'}\circ g
$$
The process $Y_t=s(X_t)$ satisfies the SDE $dY_t=\mu(Y_t)dt+dB_t$. The transition density is known in semi-closed form (see \cite{Zmirou}) and given by
$$
\pi(t,x,y)=\frac{1}{\sqrt{2\pi t}}\frac{1}{\sigma(y)}e^{-\frac{(s(y)-s(x))^2}{2t}}U_t(s(x),s(y))
$$
where $U_t(x,y)=H_t(x,y)e^{A(y)-A(x)}$, $H_t(x,y)=E(e^{-t\int_0^1h(x+z(y-x)+\sqrt{t}W_z)dz})$, $W$ is a Brownian bridge, $A$ a primitive of $\mu$ and $h=\frac{1}{2}(\mu^2+(\mu^{'})^2)$. It is then straightforward to compute the ratio
\begin{align*}
\frac{1}{\pi}\frac{\partial \pi}{\partial x}(t,x,y)&=\frac{1}{\sigma(x)}\Big(\frac{s(y)-s(x)}{t}+\frac{1}{U_t(s(x),s(y))}\frac{\partial U_t}{\partial x}(s(x),s(y))\Big)\\
&=\frac{1}{\sigma(x)}\Big(\frac{s(y)-s(x)}{t}+\frac{1}{H_t(s(x),s(y))}\frac{\partial H_t}{\partial x}(s(x),s(y))-\mu(s(x))\Big)
\end{align*}
From the boundedness assumptions of $b$ and $\sigma$ and their derivatives, there exists a constant $M$ such that $|\frac{1}{\pi}\frac{\partial \pi}{\partial x}(t,x,y)|\leq M(1+\frac{|s(y)-s(x)|}{t})$. Hence, for $0\leq s<t$
$$
E\Big|\frac{1}{\pi}\frac{\partial \pi}{\partial x}(t-s,X_s,X_t)\Big|\leq M\big(1+E\Big|\frac{s(X_t)-s(X_s)}{t-s}\Big|\big)
$$
But $s(X_t)-s(X_s)=Y_t-Y_s=\int_s^t\mu(Y_u)du+W_t-W_s$. But $\mu$ is bounded, hence
$$
E|s(Y_t)-s(Y_s)|\leq ||\mu||_{\infty}|t-s|+E|W_t-W_s|=||\mu||_{\infty}|t-s|+\sqrt{\frac{2}{\pi}}\sqrt{t-s}
$$
This proves the existence of a constant $C$ such that
$$
E\Big|\frac{1}{\pi}\frac{\partial \pi}{\partial x}(t-s,X_s,X_t)\Big|\leq C (1+\frac{1}{\sqrt{t-s}})
$$
Since $\phi(x)=C (1+\frac{1}{\sqrt{x}})$ is integrable in zero, we can apply Theorem \ref{T:diff} and conclude.
\end{proof}

\section{Applications to insider trading models}\label{section5}

We have recently seen evidence of a rather spectacular use of inside information during the trial and conviction of Raj Rajaratnam of the Galleon Group, and the subsequent conviction of his co-conspirator Rajat Gupta, abusing his board membership at Goldman Sachs. Other much reported scandals include those of Martha Stewart (of Martha Stewart Living Omnimedia), and Mark Cuban, the billionaire owner of the Dallas Mavericks (a professional basketball team). There are many stories in the media; one example is given in~\cite{Henning}. The Barclay's (and other mega banks') manipulation of LIBOR is unfolding as this paper is being written. Therefore it is reasonable to try to understand this recurrent and insidious phenomenon via mathematical modeling.

Insider trading models using stochastic calculus go back to the seminal work of A. Kyle in 1985~\cite{Kyle}, with a rigorous treatment developed soon after by K. Back~\cite{Back1},\cite{Back2}. The ideas are straightforward: If we have a filtration $(\mcf_t)_{t\geq 0}=\mbf$ that represents the collective information available to the market, there might arise entities that have extra, ``inside" information that is not publicly available. Since the insiders see more observable events than does the market in general, we can model this by using a larger filtration $(\mcg_t)_{t\geq 0}=\mbg$ that contains $\mbf$. We obtain $\mbg$ by carefully expanding $\mbf$ in such a way that all $(\Omega,\mbf,P)$ semimartingales remain semimartingales in the filtered measure space $(\Omega,\mbg,P)$; often we can also obtain the new semimartingale decomposition in the expanded filtration as well. 

To illustrate the basic ideas, let's assume we are dealing with continuous semimartingales in a Brownian paradigm, with complete markets. We also assume that the spot interest rate is $0$. On our underlying space $(\Omega,\mcf,\mbf,P)$ we have a nonnegative (continuous) price process $S$, which has a decomposition $S=S_0+M+A$ where $M$ is a continuous local martingale, and $A$ is a continuous process with paths of finite variation on compact time sets.  Moreover we assume $M_0=A_0=0$ so such a decomposition is unique. Our time interval of interest will be $[0,T]$, so we are working in a finite horizon model. If $S$ satisfies the NFLVR condition, as we will assume it does to avoid uninteresting cases, then we can find a probability measure $Q$ equivalent to $P$ (written $Q\sim P$) such that $S$ itself is a local martingale under $Q$. The measure $Q$ is called the \emph{risk neutral measure} and in this complete market it determines the fair (no arbitrage) prices of financial derivatives, such as call and put options, via expectation under $Q$. This is all well known, see for example either of~\cite{DSbook},\cite{JP1}. 

\subsection{Constructing the Risk Neutral Measure of the Insider}\label{riskneutral}
Recall we have under $P$ that $S$ decomposes as
\begin{equation}\label{5e1}
S_t=S_0+M_t+A_t \quad t\geq 0.
\end{equation}
Since the market is assumed to be complete, the risk neutral measure $Q$ is unique. Because $M$ is within the Brownian paradigm, we know by martingale representation that $M_t=\int_0^tJ_sdB_s$ for some predictable integrand $J$. Also, $S$ satisfies NFLVR, so in this case we must have that $A$ is of the form $A_t=\int_0^tH_sds$ for some predictable process $H$. We now set 
\begin{equation}\label{5e2}
\frac{dQ}{dP}=\exp\left(\int_0^T-\frac{H_s}{J_s}dB_s-\frac{1}{2}\int_0^T\frac{H^2_s}{J^2_s}ds\right).
\end{equation}
Inspired by standard theory, we let $Z$ be the unique solution of the stochastic exponential equation
\begin{equation}\label{5e3}
Z_t=1-\int_0^tZ_s\frac{H_s}{J_s}dB_s
\end{equation}
and then by Girsanov's theorem (see, e.g.,~\cite{Protter:2005}) we have that
\begin{equation}\label{5e4}
N_t=\int_0^tJ_sdB_s-\int_0^t\frac{1}{Z_s}d[Z,J\cdot B]_s\text{ is a }Q\text{ local martingale}.
\end{equation}
We make the calculation
\begin{eqnarray*}
[Z,J\cdot B]_t&=&[-Z\frac{H}{J}\cdot B,J\cdot B]_t\\
&=&-\int_0^tZ_s\frac{H_s}{J_s}J_sd[B,B]_s\\
&=&-\int_0^tZ_sH_sds
\end{eqnarray*}
Combining this calculation with~\eqref{5e4} gives that
\begin{eqnarray*}
N_t&=&\int_0^tJ_sdB_s-\int_0^t-\frac{1}{Z_s}Z_sH_sds\\
&=&\int_0^tJ_sdB_s+\int_0^tH_sds=\text{ a }Q\text{ local martingale}.
\end{eqnarray*}
Note that we now have $N=S$ which implies that $S$ is a $Q$ local martingale. Since there is only one such measure that turns $S$ into a local martingale  (the market is complete), we see that our seemingly \emph{ad hoc} definition of $\frac{dQ}{dP}$ that gives us $Q$ is exactly the right one, and the only one, that works to give us the risk neutral measure. As a caveat, we quickly add that in the above analysis, we have implicitly assumed that all stochastic integrals exist, and in particular that dividing by the process $J$ does not cause any problems. 

The point of the above calculations is that the Radon-Nikodym density $\frac{dQ}{dP}$ given in~\eqref{5e2} depends on the processes $H$ and $J$. These processes can and usually do change under an expansion of filtrations, and therefore they affect $\frac{dQ}{dP}$, changing the risk neutral measure. We denote $J^\star$ and $H^\star$ to be the $\mbg$ processes in the decomposition of $S$. The risk neutral measure changes  from $Q$ to a new measure $Q^\star$ for the insider, and it is different than it is for the market. Since derivative prices are expectations under $Q$ for the collective market using $\mbf$, and they are expectations under $Q^\star$ for the insider using $\mbg$, the result is that the insider has different derivative prices using the filtration $\mbg$, giving him a potentially tremendous advantage with which he can derive more profits, by knowing when a trade that appears neutral and fair to the market under $Q$ is actually a bargain (or is overpriced) if the price is computed for the insider under $Q^\star$. 


\begin{remark}\label{r1}
It is also possible that under $\mbg$ the process $H^\star$ does not have almost surely square integrable paths with respect to $dt$, and therefore the second integral on the right side of~\eqref{5e2} need not exist a.s.. In this case there is no risk neutral measure $Q^\star$, or at least no such measure that is equivalent to $P$, and NFLVR is violated under $\mbg$. That is, insider models can introduce arbitrage opportunities, even when there is no arbitrage in the original $(\Omega,\mbf,P)$ model. This idea has been developed for example in the articles~\cite{IPW},\cite{Imkeller}. In particular Peter Imkeller, in~\cite[Theorem 4.1]{Imkeller} shows that insider knowledge of the last time $L$  a price process following a recurrent diffusion crosses 0 before a fixed time $T$ is an arbitrage opportunity, and he makes explicit calculations to show that one does indeed lose the square integrability in the drift term under the filtration expansion, so there cannot be a risk neutral measure that is equivalent, and therefore NFLVR is violated. This is a reassuring result, if not a surprising one, because an obvious arbitrage strategy for the insider is to buy the stock immediately after time $L$ if it is above zero, and to sell it short immediately after $L$ if it is below zero. He assumes (for the sake of the calculations) that the recurrent diffusion satisfies a standard stochastic differential equation, with some restrictions on the generality of the coefficients. The choice of the level $0$ is of course arbitrary, and we could be dealing with a more general asset than a stock price, so there is no need to insist that it remain nonnegative. A similar event, without obvious arbitrage opportunities, is the second to last time a price process crosses a level, or better, the second to last time it reaches the boundary of a band with upper and lower crossing bounds; however these examples are less amenable to explicit calculations in the spirit of Imkeller.
\end{remark}

\subsection{Progressive Expansion}
What is interesting about our method as regards models of insider trading is that it allows the expansion of filtration via an ongoing progressive expansion, as more and more information comes available. While we were writing an earlier version of this paper, the world of banking provided yet another excellent example, mainly the LIBOR scandal.\footnote{LIBOR is an acronym for the London Interbank Offer Rate. It is the primary benchmark, along with the Euribor, for short term interest rates around the world. LIBOR is determined through an average of the interest rates banks must pay to one another for overnight and other loans, with maturities up to one year. It is computed after truncating the extremes. Overnight loans are made from banks with excess capital reserves to banks with insufficient capital reserves, so that banks can meet regulatory requirements on a daily basis.  The rates charged can reflect the health of the bank, as determined by the market of other banks that might lend them money.} This works well for our concept of process expansion, because as the fudging of interest rates reported to LIBOR by the banks began and became ongoing, the knowledge of what was going on gradually diffused into the market as time evolved, and what was going on was changing with time. LIBOR is important because other financial products (such as some financial derivatives, adjustable rate mortgages, many kinds of loans) are based on LIBOR, either directly or indirectly. 

Our first interesting example is related to the famous  Bessel 3 process.
\subsubsection{The Case of the Bessel 3 Process}\label{Bessel}

Let $Z$ be a Bessel 3 process, $\mathbb F$ its natural filtration, $X_t=\inf_{s>t}Z_s$ and $\mathbb G$ the progressive expansion of $\mathbb F$ with $X$. This example has been studied in detail by both Jeulin and Pitman using different techniques. 
Let $B_t=Z_t-\int_0^t\frac{ds}{Z_s}$. It is a classical result that $B$ is an $\mathbb F$~Brownian motion. Using Williams' path decomposition for Brownian motion, Pitman~\cite{Pitman} proves that $B_t-(2X_t-\int_0^t\frac{ds}{Z_s})$ is a $\mathbb G$~Brownian motion. Using filtration expansion results, Jeulin proves in \cite{Jeulin} the $\mathbb G$~semimartingale property of $B$ and provides its decomposition simultaneously. However, his technique is hard to generalise. Our approach allows us to prove the semimartingale property of $B$, albeit without finding the explicit decomposition. Nevertheless it has merit because it can be used in much more general settings as described in Theorem~\ref{T:5}.

\begin{theorem}\label{T:6}
The process $B$ is a $\mathbb G$~special semimartingale.
\end{theorem}

\begin{proof}
First $X$ is continuous increasing to infinity since $\lim_{t\rightarrow\infty}Z_t=\infty$ (see Lemma 6.20 in Jeulin). Let $\gep_n$ be a sequence of positive real numbers decreasing to zero. Therefore $\tau^n_p=\inf\{t, X_t\geq p\gep_n\}$ is increasing to infinity, for each $n$. Now, with $Y_t=2X_t-Z_t$, we have $X_t=\sup_{s\leq t}Y_s$, so that 
$$
\tau^n_p=\inf\{t, X_t\geq p\gep_n\}=\inf\{t, Y_t\geq p\gep_n\}=\sup\{t,Z_t=p\gep_n\}
$$
Therefore $(\tau^n_p)_{p\geq 1}$ satisfies Assumption \ref{A:6}. We compute now $A^n$ as of equation~(\ref{eq:3}). It is a classical result that
$$
Z^{n,p}_t=P(\tau^n_p>t\mid\mathcal F_t)=1\wedge\frac{p\gep_n}{Z_t}
$$
and Tanaka's formula implies that 
$$
M^{n,p}_t=1-p\gep_n\int_0^t1_{\{Z_s>p\gep_n\}}\frac{dB_s}{Z_s^2}
$$
Therefore since on $\{\tau^n_p<s\}$ we have that $Z_s>p\gep_n$, we obtain:
\begin{align*}
A^n_t&=\sum_{p=0}^{\infty}\int_0^t1_{\{\tau^n_p<s\leq\tau^n_{p+1}\}}\frac{d\langle B,M^{n,p+1}-M^{n,p}\rangle_s}{Z^{n,p+1}_s-Z^{n,p}_s}\\
&=\sum_{p=0}^{\infty}\int_0^t1_{\{\tau^n_p<s\leq\tau^n_{p+1}\}}\frac{\frac{p\gep_n}{Z_s^2}-1_{\{Z_s>(p+1)\gep_n\}}\frac{(p+1)\gep_n}{Z_s^2}}{\frac{(p+1)\gep_n}{Z_s}1_{\{(p+1)\gep_n<Z_s\}}+1_{\{(p+1)\gep_n\geq Z_s\}}-\frac{p\gep_n}{Z_s}}ds\\
&=\sum_{p=0}^{\infty}\int_0^t1_{\{\tau^n_p<s\leq\tau^n_{p+1}\}}\Big(1_{\{(p+1)\gep_n\geq Z_s\}}\frac{p\gep_n}{-Z_sp\gep_n+Z_s^2}+1_{\{(p+1)\gep_n<Z_s\}}\frac{-1}{Z_s}\Big)ds\\
&=\sum_{p=0}^{\infty}\int_0^t1_{\{\tau^n_p<s\leq\tau^n_{p+1}\}}\Big(-\frac{1}{Z_s}+1_{\{(p+1)\gep_n\geq Z_s\}}\frac{1}{Z_s-p\gep_n}\Big)ds
\end{align*}
Fubini's theorem implies finally that
$$
A^n_t=\int_0^t\frac{ds}{Z_s}-\sum_{p=0}^{\infty}\int_0^t1_{\{\tau^n_p<s\}}1_{\{(p+1)\gep_n\geq Z_s\}}\frac{1}{Z_s-p\gep_n}ds
$$
where we also used $1_{\{s\leq \tau^n_{p+1}\}}1_{\{Z_s\leq (p+1)\gep_n\}}=1_{\{Z_s\leq (p+1)\gep_n\}}$. Now
\begin{align*}
E&(\sum_{p=0}^{\infty}\int_0^t1_{\{\tau^n_p<s\}}1_{\{(p+1)\gep_n\geq Z_s\}}\frac{1}{Z_s-p\gep_n}ds)=\sum_{p=0}^{\infty}E(\int_0^t1_{\{\tau^n_p<s\}}1_{\{(p+1)\gep_n\geq Z_s\}}\frac{1}{Z_s-p\gep_n})ds\\
&=\sum_{p=0}^{\infty}E(\int_0^t\frac{1}{Z_s}1_{\{Z_s>p\gep_n\}}1_{\{Z_s\leq (p+1)\gep_n\}})ds=E(\int_0^t\frac{1}{Z_s}\sum_{p=0}^{\infty}1_{\{p\gep_n<Z_s\leq (p+1)\gep_n\}}ds)=E(\int_0^t\frac{ds}{Z_s})
\end{align*}
where the second equality follows because the $\mathbb F$~optional projection of $1_{\{\tau^n_p\leq \cdot\}}$ is $(1-\frac{p\gep_n}{Z_t})^{+}$. It remains to use Theorem \ref{T:5} to conclude. 
\end{proof}

We find this example quite interesting, because we see in the proof of Theorem~\ref{T:6} that each process $(A^n_s)_{s\geq 0}$ a.s. has paths that are absolutely continuous with respect to $ds$, yet as we see in the limit, thanks to the results of Jeulin and Pitman, that the $\mbg$ decomposition 
$B$ is given by 
\begin{equation}\label{Be1}
B_t=\left(B_t-(2X_t-\int_0^t\frac{ds}{Z_s})\right)+(2X_t-\int_0^t\frac{ds}{Z_s}).
\end{equation}
The finite variation term of~\eqref{Be1} is $2X_t-\int_0^t\frac{ds}{Z_s}$.  We note that the process $X$ is non decreasing but $dX_s$ has support on a random set which has Lebesgue measure $0$ a.s. Due to the presence of a singular term in the decomposition~\eqref{Be1} we cannot find an equivalent probability measure that turns $B$ into a $\mbg$ (local) martingale. So we are in the situation where each approximating term is well behaved, but in the limit the process we are after cannot be transformed into a local martingale and this example shows that we have NFLVR, implying the presence of arbitrage opportunities. Intuitively, this filtration expansion introduces arbitrage opportunities into the market where an insider discovers the extra information $X_t$ progressively and obtains an arbitrage opportunity that is hidden from the rest of the market. Brownian motion is not a good model of a stock price (for example it does not remain positive) but an extension of the above could apply, for example, to a model of what transpired with the Galleon Group, previously mentioned in the introduction to this section (Section~\ref{section5}), and in Section~\ref{questions}. See~\cite{MansuyYor} for references concerning this example, and also~\cite{Yor} for a fine analysis of many aspects of the Bessel (3) process and related processes.


\subsubsection{The Case of Transient Diffusions}\label{Subsection6.1.2}

In the preceding section (Section~\ref{Bessel}) we showed one individual process remained a semimartingale in the expanded filtration. It is more interesting, and a more powerful result, to show all $\mbf$ semimartingales remain semimartingales in $\mbg$, albeit with different decompositions in general. This is known as Hypothesis $(H^\prime)$. This might be too much to ask with these rather general filtration expansions, but we can provide a sufficient condition under which a class of semimartingales in $\mbf$ will remain semimartingales in $\mbg$. In so doing we extend the results of Section~\ref{Bessel}.
 
Let $R_t$ be a transient diffusion with values in $\mathbb R^{+}$, which has $\{0\}$ as entrance boundary. Let $s$ be a scale function for $R$, which we can choose such that 
$$
\lim_{x\rightarrow 0}s(x)=-\infty \qquad \text{and} \qquad \lim_{x\rightarrow \infty}s(x)=0
$$
Let $\mathbb F$ be the natural filtration of $R$. Nikeghbali~\cite{N} studied progressive filtration expansions of $\mathbb F$ with last exit times of such diffusions. Define $X$ to be the remaining infimum of $R$, i.e.~$X_t=\inf_{s>t}R_s$ and $\mathbb G$ the progressive expansion of $\mathbb F$ with $X$. We provide a sufficient condition for some $\mathbb F$ martingales to remain $\mathbb G$~semimartingales. The process $M_t=-s(R_t)$, which is well known to be be an $\mathbb F$ local martingale, plays a key role. The next theorem is the main result of this section.

\begin{theorem}\label{T:7}
Let $N$ be an $\mathbb F$ martingale such that $\sup_{0\leq s\leq T}|N_s|$ and $\int_0^T\frac{|d\langle N,M\rangle_s|}{M_s}$ are integrable. Then $N$ is a $\mathbb G$~semimartingale.
\end{theorem}

\begin{proof}
Define the $\mathbb F$~honest random times $\sigma_y=\sup\{t, R_t=y\}$. Let $Y_t=2X_t-R_t$, then $X_t=\sup_{s\leq t}Y_s$ and the random times 
$$
\tau^n_p:=\inf\{t, X_t\geq p\gep_n\}=\inf\{t,Y_t\geq p\gep_n\}=\sup\{t, R_t=p\gep_n\}=\sigma_{p\gep_n}
$$
are $\mathbb F$~honest and $(\tau^n_p)_{p\geq 1}$ satisfies Assumption \ref{A:6}. To use Theorem \ref{T:5}, we need to compute $A^n$ as defined in (\ref{eq:3}). It is proved in \cite{N} that
$$
P(\sigma_y>t\mid\mathcal F_t)=\frac{s(R_t)}{s(y)}\wedge 1=1-\frac{1}{s(y)}\int_0^t1_{\{R_u>y\}}dM_u+\frac{1}{2s(y)}L^{s(y)}_t
$$
where $L^{s(y)}$ is the local time of $s(R)$ at $s(y)$. Introduce 
$Z^{n,p}_t=P(\tau^n_p>t\mid\mathcal F_t)=\frac{s(R_t)}{s(p\gep_n)}\wedge 1$ and the martingale part in its Doob Meyer decomposition $M^{n,p}_t=1-\frac{1}{s(p\gep_n)}\int_0^t1_{\{R_u>p\gep_n\}}dM_u$. Therefore
$$
A^n_t=\sum_{p=0}^{\infty}\int_0^t1_{\{\tau^n_p<s\leq \tau^n_{p+1}\}}\frac{\frac{-1}{s((p+1)\gep_n)}1_{\{R_s>(p+1)\gep_n\}}+\frac{1}{s(p\gep_n)}}{1_{\{R_s\leq (p+1)\gep_n\}}+1_{\{R_s>(p+1)\gep_n\}}\frac{s(R_s)}{s((p+1)\gep_n)}-\frac{s(R_s)}{s(p\gep_n)}}d\langle N,M\rangle_s
$$
where we used that $R_s>p\gep_n$ on $\{\tau^n_p<s\}$, and the fact that $-s$ is positive non increasing by construction. Basic algebraic manipulations give
\begin{align*}
A^n_t&=\sum_{p=0}^{\infty}\int_0^t1_{\{\tau^n_p<s\leq \tau^n_{p+1}\}}\Big(\frac{1}{-s(R_s)}+1_{\{R_s\leq (p+1)\gep_n\}}\big(\frac{1}{s(R_s)}+\frac{1}{s(p\gep_n)-s(R_s)}\big)\Big)d\langle N,M\rangle_s\\
&=\int_0^t\frac{d\langle N,M\rangle_s}{M_s}+\sum_{p=0}^{\infty}\int_0^t1_{\{\tau^n_p<s\}}1_{\{R_s\leq (p+1)\gep_n\}}\frac{s(p\gep_n)}{s(p\gep_n)-s(R_s)}\frac{d\langle N,M\rangle_s}{s(R_s)}
\end{align*}
where the last equality follows from $1_{\{s\leq\tau^n_{p+1}\}}1_{\{R_s\leq (p+1)\gep_n\}}=1_{\{R_s\leq (p+1)\gep_n\}}$. Now 
\begin{align*}
E(\int_0^T|dA^n_s|)&\leq E(\int_0^T\big|\frac{d\langle N,M\rangle_s}{M_s}\big|)\\
&+\sum_{p=0}^{\infty}E(\int_0^T1_{\{\tau^n_p<s\}}1_{\{R_s\leq (p+1)\gep_n\}}\big|\frac{s(p\gep_n)}{s(R_s)(s(p\gep_n)-s(R_s))}\big||d\langle N,M\rangle_s|)\\
&\leq E(\int_0^T\big|\frac{d\langle N,M\rangle_s}{M_s}\big|)+\sum_{p=0}^{\infty}E(\int_0^T1_{\{p\gep_n<R_s\leq (p+1)\gep_n\}}\frac{1}{M_s}|d\langle N,M\rangle_s|)\\
&\leq 2 E\big(\int_0^T\frac{1}{M_s}d\langle N,M\rangle_s\big)
\end{align*}
where the second inequality uses that the $\mathbb F$~optional projection of $1_{\{\tau^n_p\leq \cdot\}}$ is given by $\big(1-\frac{s(R_s)}{s(p\gep_n)}\big)^{+}$ and the monotonicity of $s$. Theorem \ref{T:5} allows to conclude.
\end{proof}

Unfortunately we do not have any results of Pitman, Jeulin, or even Nikeghbali to help us determine what the $\mbg$ semimartingale decomposition of $N$ actually is. So we are unable to determine, in general, if such an expansion introduces arbitrage, as it does in the case of the Bessel 3 process, as show in Section~\ref{Bessel}. Nor can we explicitly calculate the risk neutral measure. Both of these results must await further research, and they are certainly of intrinsic interest. Intuitively, this is an extension of the Bessel 3 example where a singular term appears in the $\mbg$ decomposition, so we would not be surprised if that pathology occurs much more generally.  One can consult~\cite{MansuyYor} for references related to this example. 

\subsubsection{A Suggestive Example}\label{ASE}
We study here an example derived from elementary calculations, without using theory. We let $W$ denote a standard Brownian motion, or Wiener process, with natural filtration $\mbf$ satisfying the usual hypotheses. We also let $V$ be another standard Wiener process, independent of $W$ (and of $\mbf$). We will expand the filtration $\mbf$ dynamically in the style of this paper, with the process
\begin{equation}\label{se1}
X_t=W_1+\gep V_{1-t}
\end{equation}
We define
\bee\label{10e15bis}
\mch_t=\mcf_t\vee\gs(W_1+\gep V_{1-s}; s\leq t)=\mcf_t\vee\gs(X_s:s\leq t).
\eee
We wish to show the following result, obtained with the help of Jean Jacod, to whom we are grateful.
\begin{theorem}\label{set1}
Let $H$ be predictable with $\int_0^1H_s^2ds<\infty$ a.s. Define $M_t=\int_0^tH_sdW_s$, an $\mbf$ local martingale. If $H$ is a.s. of the order $H_s=\frac{1}{1-s}^{1/2+\alpha}$ with $\alpha<\frac{1}{2}$ then $M$ remains a semimartingale in $\mbh$, and has decomposition
\begin{equation}
A^\mbh_t=-\int_0^tH_s\frac{X_s-W_s}{(1+\gep^2)(1-s)}ds.
\end{equation}
\end{theorem}
\begin{proof}
First we review some standard calculations, which are nevertheless not trivial. Let $0\leq t\leq T\leq 1$. Then
\begin{eqnarray*}
E(W_T\vert W_1,W_t)&=&a(W_1-W_t)+bW_t\text{ due to the linearity,}\\
&&\text{and that }\sigma(W_1,W_t)=\sigma(W_1-W_t,W_t);\\
E(W_T(W_1-W_t))&=&a(1-t)=T-t\quad\Rightarrow\quad a=\frac{T-t}{1-t}\\
E(W_tW_t)&=&t=bt\quad\Rightarrow \quad b=1, \text{ hence }\\
E(W_T\vert\mcf_t\vee\sigma(W_1))&=&W_t+\frac{T-t}{1-t}(W_1-W_t).
\end{eqnarray*}

\bigskip

\noindent  We have that $W$ is a semimartingale for $\mbg$, where $\mcg_t=\mcf_t\vee\sigma(W_1)$, with decomposition $W=M+A$, and 
$$
A_t^\mbg=-\int_0^t\frac{W_1-W_s}{1-s}ds
$$
Note that $\mbf\subset\mbg\subset\mbh$. We use the superscript notation $A^\mbg$ and $A^\mbh$ to denote relative to which filtration we are calculating the finite variation process $A$. We make a calculation of the expected total variation, to get (where ``$\var$" denotes total variation, and not variance):
$$
E(\var(A^\mbg_t))=\int_0^t\frac{E(\vert W_1-W_s\vert)}{1-s}ds=\sqrt{\frac{2}{\pi}}\int_0^t\frac{1}{\sqrt{1-s}}ds<\infty.
$$

\bigskip

\noindent Now we let $\mch_t=\mcf_t\vee\sigma(W_1+\gep V_{1-s}, s\leq t)$.
Using the above calculations we have
\begin{eqnarray*}
E(W_T\vert\mch_t)&=&E(W_T\vert W_t,W_1+\gep V_{1-t})\\
&=&bW_t+a(W_1-W_t+\gep V_{1-t}),\text{ again by linearity, and since }\\
E(W_TW_t)&=&bt=t, \quad\text{ we have }b=1,\text{ and }\\
E(W_T(W_1-W_t+\gep V_{1-t})&=&T-t=a\left((1-t)+\gep^2(1-t)\right)\quad\Rightarrow\\
E(W_T\vert\mch_t)&=&W_t+\frac{T-t}{(1+\gep^2)(1-t)}(W_1-W_t+\gep V_{1-t})\\
&=&W_t+\frac{T-t}{(1+\gep^2)(1-t)}(X_t-W_t).
\end{eqnarray*}
Note that 
$$
X_s-W_s=W_1-W_s+\gep V_{1-s}\law \sqrt{1+\gep^2}W^\prime_{1-s}
$$
for another Wiener process $W^\prime$, and where $\law$ denotes equality in law. In this context we have that our process $A$ becomes
$$
A^\mbh_t=-\int_0^t\frac{X_s-W_s}{(1+\gep^2)(1-s)}ds\law -\int_0^t\frac{W^\prime_{1-s}}{\sqrt{1+\gep^2}(1-s)}ds.
$$

\bigskip

\noindent Next we replace our process $W$ with a local martingale of the form $M_t=\int_0^tH_sdW_s$, as stated in the hypotheses of the theorem. The above reasoning gives us that (in the $\mbg=(\mcf_t\vee\sigma(W_1))_{t\geq 0}$ paradigm):
$$
A^\mbg_t=-\int_0^tH_s\frac{W_1-W_s}{1-s}ds,\text{ where of course }\int_0^tH_s^2ds<\infty.
$$
Then 
$$
E(V_{1-t}M_t)=aE(\int_0^t H_s\frac{\sqrt{1-s}}{1-s}ds)=aE(\int_0^t H_s\frac{1}{\sqrt{1-s}}ds)
$$
but $\int_0^1H^2_sds<\infty$ does not imply that $\int_0^1\frac{\vert H_s\vert}{\sqrt{1-s}}ds<\infty$. It does indeed imply it for $H_s$ of the form $H_s=\frac{1}{1-s}^{1/2+\alpha}$ with $\alpha<\frac{1}{2}$, but it does not work for $H_s=\frac{1}{\sqrt{1-s}\vert \ln{(1-s)\vert}}$. Therefore we obtain the semimartingale property for any predictable process $H$ and the (local) martingale $M_t=\int_0^tH_sdW_s$ as long as $\vert H\vert$ is of the order  $\vert H_s\vert\propto\frac{1}{\sqrt{1-s}}$.\end{proof}

This result follows also from our Theorem~\ref{T:semimgX} and the proof is similar to that of Theorem~\ref{T:diff}. We sketch it here briefly.

\begin{proof}
First, we can prove that $(X, \mathbb F)$ satisfies Assumption~\ref{A:J} and that for each $0\leq i \leq n$ and $(t^n_i)_{0\leq i\leq n}$ subdivision of $[0,1]$, the $\mathcal F_t$-conditional density of $(X_{t^n_0}, X_{t^n_1}-X_{t^n_0}, \ldots, X_{t^n_i}-X_{t^n_{i-1}})$ is given by
$$
p^{i,n}_t(x_0, \ldots, x_i) = g\big(\frac{\sum_{k=0}^ix_k - W_t}{\sqrt{1-t+\gep^2(1-t^n_i)}}\big) \prod_{k=0}^{i-1}\frac{1}{\sqrt{t^n_{k+1}-t^n_k}}g(\frac{-x_{k+1}}{t^n_{k+1}-t^n_k})
$$
where g is the gaussian density. Second, with $M=W$, equation (\ref{eq:Ain}) translates into
$$
A^{i,n}_t = \int_0^t \frac{X_{t^n_i}-W_s}{(1-s)+\gep^2(1-t^n_i)}ds
$$
and equation (\ref{eq:An}) becomes
$$
A^{(n)}_t=\sum_{i=0}^n1_{\{t^n_i\leq t<t^n_{i+1}\}}\Big(\sum_{k=0}^{i-1}\int_{t^n_k}^{t^n_{k+1}}\frac{X_{t^n_k}-W_s}{(1-s)+\gep^2(1-t^n_k)}ds+\int_{t^n_i}^t\frac{X_{t^n_i}-W_s}{(1-s)+\gep^2(1-t^n_i)}\Big)
$$

Third, we prove that the total variation of $A^{(n)}$ can be bounded uniformly in $n$ :
$$
E\big(\int_0^1|dA^{(n)}_s|\big)\leq\sum_{k=0}^n\int_{t^n_k}^{t^n_{k+1}}E\Big|\frac{X_{t^n_k}-W_s}{(1-s)+\gep^2(1-t^n_k)}\Big|ds
\leq \sqrt{\frac{2}{\pi}}\int_{0}^{1}\frac{ds}{(1-s)(1+\gep^2)} < \infty
$$

Therefore we can apply Theorem \ref{T:semimgX} to conclude that $W$ is an $\mathbb H$~semimartingale. Finally $A^{(n)}$ converges in probability to the process $A$ given by
$$
A_t=\int_0^t\frac{X_s-W_s}{(1-s)(1+\gep^2)}ds
$$
We can conclude as in the proof of Theorem \ref{T:diff} that $W-A$ is an $\mathbb H$ Brownian motion.
\end{proof}

What this means for insider trading is that since the finite variation terms a.s. has paths that are absolutely continuous with respect to $d[W,W]_t=dt=$ Lebesgue measure, \emph{if we have enough integrability satisfied}, we can use a Girsanov transformation to construct the risk neutral measure for the insider, as we outlined in Section~\ref{riskneutral}.

\subsection{A Current Example from Industry}

Recently the attorney general of New York State, Eric Schneiderman, has undertaken an investigation of high frequency trading (HFT). By HFT we mean firms that have co-located around the computers processing trades (in the case of the NewYork Stock Exchange this is in Mahwah, NJ). See~\cite{WA}. The claims are that the HFTs effectively have inside information. One way this could happen is a consequence of the analysis presented in a recent paper of Jarrow and Protter~\cite{JP3}. In brief, by the systematic use of IOC (immediate or cancel) orders, the HFT traders are able to construct a representation of the current state of the order book. They are then able to take the liquidity profits via an exploitation both of market orders \emph{and also of limit orders}. These profits were formerly taken by large institutional traders, at the expense of small and unsophisticated traders.  By doing this they can profit in the place of institutional traders, and also at their expense. But one can take the analysis further, and by effectively front running the limit orders of institutional traders they ensure that limit orders are executed almost exclusively at their limits, with the HFT traders pocketing the spread between the market price and the limit order limit price. See~\cite{JP3} for details. 

The ``inside information'' attributed to the HFTs comes primarily from their real time understanding of the limit order book, obtained via the systematic use of the IOC orders technique.  If one can see the limit order book, one gains insight into the very near future direction of the stock price. This type of inside information lends itself to the model presented in Section~\ref{ASE}. We now discuss how this works.

We suppose given two independent standard Brownian motions $W$ and $V$ on a space $(\Omega,\mcf,P,\mbf)$. We assume we have a stock price given by
\bee\label{10e17}
dZ_s=\gs(s,Z_s)dW_s+b(s,Z_s)ds\text{ with }Z_0=1
\eee
We suppose we can see the direction of the price via the limit order book (LOB). Obviously this should be impossible in $\mbf$ since in any event $Z$ is strong Markov with respect to $\mbf$. However we incorporate this new information via a filtration enlargement, and this destroys the Markov property. Indeed, suppose we can see a future evolution via the LOB. We cannot expect actually to know a future value of the stock, but we can have a better guess at the future value than a competitive trader who only sees $\mbf$ observable events. Therefore it is not unreasonable to assume we can see the future of $W$ (from which we can infer the future value of $Z$), but corrupted by a small amount of noise.  We use the terms given in~\eqref{se1} and~\eqref{10e15bis}. More precisely, the HFT trader would want to know the information of $W_1$ in this example, which would then give him information regarding the future evolution of the stock price $Z$. He cannot tell with certainty what $W_1$ is (since it is in the future), but he can guess at it based on the order book. But the order book information is itself noisy since it is dependent on future human behavior, and -- for example -- limit orders can be cancelled before execution, and/or have been placed there to deceive, rather than with sincerity. We model this via the inclusion of random noise in the LOB.  We use the noise corruption term coming from $\gep$ times the independent Wiener process $V$. The multiplication by $\gep$ represents the idea that while there is noise in the LOB, it is of very small variance. So what the HFT sees at time $t$ is $X_t=W_1+\gep V_{1-t}$, and the more the HFT can see, the better he or she can try to filter out the noise as much as possible. Note that this does not lead to an arbitrage opportunity \emph{a fortiori}, but what is known as a statistical arbitrage opportunity. This is explained in~\cite{JP3}, for example. 

To continue the analysis, we  apply Theorem~\ref{set1} to conclude in the larger filtration $\mbh$ we have that the $Z$ of~\eqref{10e17} now satisfies
\bee\label{10e19}
Z_t=1+\int_0^t\gs(s,Z_s)dW_s+\int_0^tb(s,Z_s)ds-\int_0^t\gs(s,Z_s)\frac{X_s-W_s}{(1+\gep^2)(1-s)}ds
\eee
We have that \emph{the insiders see a different dynamic evolution than does the rest of the market}. 

This has a particular consequence when we calculate the risk neutral measures of the general market and that of the insider. Recall that the risk neutral measures (also known as equivalent local martingale probability measures) are typically used to price contingent claims (such as options) in mathematical finance theory. 

We begin with the routine computation of the risk neutral measure where $Z$ is as in~\eqref{10e17}, and we are working with the filtration $\mbf$. To compute the risk neutral measure we first use a Girsanov argument, where we find a change of measure $Q$ such that $Z$ becomes a local martingale under $Q$. To do this, we need to find a process $H$ such that if $U$ satisfies the equation
\bee\label{10e20}
dU_t=U_tH_tdW_t; \quad U_0=1
\eee
and if $Q$ is given by $dQ=U_1dP$ then $Z$ becomes a $(Q.\mbf)$ local martingale. By Girsanov's theorem (see for example~\cite{Protter:2005}) we have that under $Q$ the decomposition of $Z$ is:
\begin{eqnarray}\label{10e21}
Z_t&=&\left(\int_0^t\gs(s,Z_s)dW_s-\int_0^t\frac{1}{U_s}d[U,W]_s\right)+\left(\int_0^tb(s,Z_s)ds+\int_0^t\frac{1}{U_s}d[U,W]_s\right)\\
&=&\left(\int_0^t\gs(s,Z_s)dW_s-\int_0^tH_sds\right)+\left(\int_0^tb(s,Z_s)ds+\int_0^tH_sds\right).
\end{eqnarray}
In order to have the drift term become 0, we need to choose $H_t=-\frac{b(Z_t)}{\gs(Z_t)}$. However we also need to ensure that $Q$ is a true probability measure. We can do this using Novikov's criterion, if $E(\exp(\frac{1}{2}\int_0^t\frac{b(s,Z_s)^2}{\gs(s,Z_s)^2}ds))<\infty$. This is of course an assumption on the drift coefficient $b$ of the original equation~\eqref{10e17} which we are free to make when setting up the model. We know that under $Q$ the process $Z$ of~\eqref{10e17} satisfies the following equation
\bee\label{10e22}
dZ_s=\gs(Z_s)d\beta_s
\eee
where $\beta$ is a Brownian motion under $Q$ (indeed, from the above analysis we see that $\beta_t=W_t-\int_0^t(-\frac{b(s,Z_s)}{\gs(s,Z_s)})ds$).

The computation of the risk neutral measure $R$ using the filtration $\mbh$ is more delicate. Again we assume $Z$ follows the equation~\eqref{10e17}. By Theorem~\ref{set1} we have, in $\mbh$, $Z$ satisfies 
\bee\label{10e23}
Z_t=1+\int_0^t\gs(s,Z_s)dW_s+\int_0^tb(s,Z_s)ds-\int_0^t\gs(s,Z_s)\frac{X_s-W_s}{(1+\gep^2)(1-s)}ds.
\eee
For equation~\eqref{10e23} to make sense, by Theorem~\ref{set1} we need to have 
\bee\label{10e24}
\gs(s,Z_s)\asymp\left(\frac{1}{1-s}\right)^{1/2+\ga}\text{ where }\ga<\frac{1}{2}
\eee
This is a serious restriction on the modeling of the stock price, but nonetheless one that we are allowed to make. But now we want to calculate the risk neutral measure. We first proceed via a Girsanov transformation: We let $dR=UdP$ where $U$ satisfies an equation of the form~\eqref{10e20}. In analogy with our previous (albeit simpler) calculation, we get
\bee\label{10e25}
H_s=-\left(\frac{b(s,Z_s)-\gs(s,Z_s)\frac{X_s-W_s}{(1+\gep^2)(1-s)}}{\gs(s,Z_s)}\right)
\eee
This is already bad enough, but now we have to verify that $R$ is a true probability measure (and not a sub probability measure).  As before a sufficient condition comes from Novikov's criterion: It suffices to have that 
\bee\label{10e26}
E(\exp(\int_0^tH_s^2ds))<\infty
\eee
But we do not in general have that~\eqref{10e26} satisfied, since we have a problem at the pole $s=1$. We can solve this either by imposing a condition on $\gs(s,z)$ such that it vanishes at $s=1$, or by not working on the half open time interval $[0,1)$. Since the former seems unreasonable, we prefer the latter. This can be finessed by a new concept known as \emph{local NFLVR}, as explained in the thesis of Roseline Bilina-Falafala~\cite{RB}, or in~\cite{BJP}.



\end{document}